\documentclass[11pt]{amsart}

\usepackage{latexsym}
\usepackage{amsmath}
\usepackage{amsfonts}
\usepackage{amssymb}
\usepackage{amsthm}
  \newcommand{\LE}{\mathcal{LE}}
   \renewcommand{\H}{\mathcal{H}}

\newcommand{\gG}{\Gamma}
\newcommand{\C}{\mathbb{C}}
\newcommand{\T}{\mathbb{T}}
\newcommand{\Q}{\mathbb{Q}}
\newcommand{\R}{\mathbb{R}}
\newcommand{\E}{\mathbb{E}}
\newcommand{\N}{\mathbb{N}}

\newcommand{\Z}{\mathbb{Z}}

\newcommand{\norm}[1]{\left\Vert #1\right\Vert}

\theoremstyle{plain}
\newtheorem{theorem}{Theorem}[section]
\newtheorem{lemma}[theorem]{Lemma}
\newtheorem{proposition}[theorem]{Proposition}

\newtheorem*{theoremA'}{Theorem A'}
\newtheorem*{theoremB'}{Theorem B'}
\newtheorem*{theoremC'}{Theorem C'}

\newtheorem*{theorem*}{Theorem}

\newtheorem{corollary}[theorem]{Corollary}
\newtheorem*{conjecture}{Conjecture}

\theoremstyle{definition}

\newtheorem{example}[theorem]{Example}

\theoremstyle{remark}

\newtheorem*{remark}{Remark}
\newtheorem*{remarks}{Remarks}

\begin{document}

\title{Equidistribution of sparse sequences on nilmanifolds}
\author{Nikos
  Frantzikinakis}
\address[Nikos  Frantzikinakis]{Department of Mathematics\\
  University of Memphis\\
  Memphis, TN \\ 38152 \\ USA } \email{frantzikinakis@gmail.com}

\begin{abstract}
We study equidistribution properties of nil-orbits $(b^nx)_{n\in\N}$
when the parameter $n$ is restricted to the range of some sparse sequence
 that is not necessarily polynomial. For example, we show that if $X=G/\Gamma$ is a
  nilmanifold, $b\in G$
 is an ergodic nilrotation, and $c\in \R\setminus \Z$ is positive,  then the sequence $(b^{[n^c]}x)_{n\in\N}$
 is equidistributed in  $X$ for every $x\in X$. This is also the case when  $n^c$ is replaced with $a(n)$, where $a(t)$
 is a function that belongs to some Hardy field, has
 polynomial growth,  and stays logarithmically away from polynomials,
  and when it  is replaced with a random  sequence of  integers with sub-exponential growth.
  Similar results have been established by  Boshernitzan when  $X$ is the circle.
\end{abstract}

\thanks{The  author was partially supported by NSF grant
  DMS-0701027.}


\keywords{Homogeneous space, nil{\bf man}if{\bf ol}d,  equid{\bf is}tribution, Hardy field. }


\maketitle

\setcounter{tocdepth}{1}
\tableofcontents

\section{Introduction and main results}

\subsection{Motivation}
A nilmanifold  is a homogeneous space $X=G/\Gamma$ where  $G$ is a nilpotent Lie group,
and $\Gamma$ is a discrete cocompact subgroup of $G$.
For $b\in G$  and  $x=g\Gamma\in X$ we  define $bx=(bg)\Gamma$.
In recent years it has become clear that studying equidistribution properties of
nil-orbits $(b^nx)_{n\in\N}$, and their subsequences,   is a central problem, with applications
to various areas
of mathematics
 that include  combinatorics (\cite{BHK}, \cite{Z2},
\cite{FK2}, \cite{Fr2},  \cite{BLL}, \cite{FLW}, \cite{FW2}),
  ergodic theory
 (\cite{CL88},  \cite{HK1}, \cite{Z1},  \cite{HK2}, \cite{L3},
 \cite{FK}, \cite{FK1}, \cite{HK3}), number theory
  (\cite{AGH}, \cite{Haa92}, \cite{BL}, \cite{GT1}),  and probability theory (\cite{Fr1}).

It is well known that for every $b\in G$ and $x\in X$ the sequence  $(b^nx)_{n\in\N}$
is equidistributed in some nice algebraic set  (\cite{Ra}, \cite{Les}, \cite{L2}), and
this  is also the case when  the  parameter $n$ is restricted to the range of some polynomial
with integer coefficients (\cite{S}, \cite{L2}), or the set of prime numbers (\cite{GT1}).
 Furthermore, very recently,  quantitative equidistribution results for  polynomial nil-orbits
  have been established (\cite{GT}) and used as part of an ongoing project
  to find
  asymptotics for the number of arithmetic progressions in the set of prime numbers
  (\cite{GT-1}).
  These quantitative estimates  will also play a crucial  role in the present article.

The main objective of this article is to  study equidistribution properties of nil-orbits
$(b^nx)_{n\in\N}$ when the  parameter $n$ is restricted to some sparse sequence of integers
 that is not necessarily polynomial.
For example,  we shall show  that if $b\in G$ has a dense orbit  in $X$, meaning $\overline{(b^n\Gamma)}_{n\in\N}=X$,    then for
every $x\in X$ the sequences
\begin{equation}\label{E:examples}
(b^{[n^{\sqrt{3}}]}x)_{n\in\N}, \quad (b^{[n \log n]}x)_{n\in\N}, \quad (b^{[n^{2}\sqrt{2}+n\sqrt{3}]}x)_{n\in\N},
\quad (b^{[n^{3}+(\log{n})^2]}x)_{n\in\N}, \quad (b^{[(\log(n!))^k]}x)_{n\in\N},
\end{equation}
are  all equidistributed in $X$. Furthermore, using a probabilistic construction we shall exhibit
examples of sequences with super-polynomial growth for which analogous equidistribution results hold
(explicit such examples are  not known). Let us remark at this point, that since we shall work with sparse sequences
of times taken along sequences whose range has typically negligible intersection with the range of polynomial sequences,  our results cannot be immediately deduced from known equidistribution results along polynomial sequences.

We shall also study equidistribution properties involving several nil-orbits. For example, suppose that  $c_1,c_2,\ldots,c_k$ are distinct non-integer real numbers,
all greater than $1$,  and $b\in G$ has a dense orbit in $X$. We shall show that the sequence
$$
(b^{[n^{c_1}]}x_1,b^{[n^{c_2}]}x_2,\ldots, b^{[n^{c_k}]}x_k)_{n\in\N}
$$
is equidistributed in $X^k$ for every $x_1,x_2,\ldots,x_k\in X$.

In a nutshell, our approach is to use the Taylor expansion of a function $a(t)$ to partition the range of
the sequence $([a(n)])_{n\in\N}$
into approximate polynomial blocks of fixed degree, in such a way that
 one can give useful quantitative estimates for the corresponding ``Weyl type'' sums. In order to carry out this plan,
  we found it very  helpful to deal with the sequence $(a(n))_{n\in\N}$ first, thus leading us to study   equidistribution
  properties of the
  sequence $(b^{a(n)}x)_{n\in\N}$, where $b^s$ for $s\in\R$  can be defined appropriately.

Before giving the exact results, let us also mention that an additional motivation
for our study
 is  the various potential applications in ergodic theory and combinatorics.
 This direction of research  has already proven fruitful; very recently  in \cite{FW2}  equidistribution results on
  nilmanifolds played a  key role in establishing a Hardy field refinement of Szemer\'edi's theorem
  on arithmetic progressions
  and related multiple recurrence results in ergodic theory.
 However, in  that article, equidistribution properties involving only  conveniently chosen subsequences
  of the sequences in question  were studied. The  problem of studying
  equidistribution properties  of the full range of non-polynomial sequences, like those in  \eqref{E:examples},
  is more delicate, and is addressed for the first time in the present article. This turns out to be a
   crucial step   towards
 an in depth  study of  the limiting behavior of multiple ergodic averages of
 the form
 $$
 \frac{1}{N}\sum_{n=1}^N T^{[a_1(n)]}f_1\cdot \ldots\cdot T^{[a_\ell(n)]}f_\ell,
 $$
 where $(a_i(n))_{n\in\N}$ are real valued sequences that satisfy some
 regularity conditions. The remaining steps of this project  will be completed
 in a forthcoming paper (\cite{Fr3}).

 \subsection{Equidistribution results}\label{SS:equidistribution}
Throughout the article we
are going to  work with the class of real valued functions $\H$ that
belong to some Hardy field (see Section~\ref{SS:Hardy} for details).
 Working within the class  $\H$  eliminates several
 technicalities  that would otherwise obscure the transparency of  our results and the main ideas of their proofs.
  Furthermore,  $\H$
 is a rich enough class to enable one to  deal, for example, with  all the sequences considered in \eqref{E:examples}.

In various places we
  evaluate an element $b$ of a connected and simply connected nilpotent Lie group $G$
on some real power $s$. In Section~\ref{SS:nil} we explain why this operation is legitimate.


 When writing $a(t)\prec b(t)$ we mean $a(t)/b(t)\to 0$ as $t\to +\infty$. When writing $a(t)\ll b(t)$ we mean
  that $|a(t)|\leq C |b(t)|$ for some constant $C$ for every large $t$. We also say that a function $a(t)$
 has \emph{polynomial growth} if $a(t)\prec t^k$ for some $k\in \N$.

\subsubsection{A single nil-orbit}
If $(a(n))_{n\in\N}$ is a sequence of real numbers, and $X=G/\Gamma$ is a nilmanifold, with $G$  connected
 and simply connected, we say that the sequence $(b^{a(n)}x)_{n\in\N}$ is \emph{equidistributed} in a sub-nilmanifold $X_b$ of $X$,
 if for every $F\in C(X)$ we have
$$
\lim_{N\to\infty}\frac{1}{N}\sum_{n=1}^N F(b^{a(n)}x)=\int F\ dm_{X_b}
$$
where $m_{X_b}$ denotes the normalized Haar measure on $X_b$.
Similarly, if the sequence $(a(n))_{n\in\N}$ has integer values,  we can  define
a notion of equidistribution on every nilmanifold $X$, without imposing any connectedness assumption on $G$ or $X$.

A sequence $(a(n))_{n\in\N}$ of real numbers is \emph{pointwise good for nilsystems}
if for every  nilmanifold $X=G/\Gamma$, where $G$ is connected and simply connected,
and  every $b\in G$, $x\in X$, the sequence
$(b^{a(n)}x)_{n\in\N}$ \emph{has a limiting distribution}, meaning,
for every $F\in C(X)$ the limit
$
\lim_{N\to\infty}\frac{1}{N}\sum_{n=1}^N F(b^{a(n)}x)
$
exists.

We remark  that for sequences $(a(n))_{n\in\N}$ with integer values,
the connectedness assumptions
 of the previous definition are superficial.
   Using the lifting argument of  Section~\ref{SS:nil}, one  sees that  if a sequence of integers $(a(n))_{n\in\N}$ is pointwise good for nilsystems,   then for every  nilmanifold $X=G/\Gamma$, $b\in G$, and $x\in X$, the sequence
$(b^{a(n)}x)_{n\in\N}$ has a limiting distribution.

Our first result
 gives necessary and sufficient conditions for Hardy sequences of polynomial growth to be pointwise good
 for nilsystems.
\begin{theorem}\label{T:A}
Let $a\in \H$ have polynomial growth.

 Then the  sequence $(a(n))_{n\in\N}$ (or the sequence  $([a(n)])_{n\in\N}$)
is pointwise good for nilsystems if and only if one of the following conditions holds:
\begin{itemize}
 \item   $|a(t)-cp(t)|\succ \log t$ for every $c\in\R$ and  every $p\in \Z[t]$; \text{ or }

 \item $a(t)-cp(t)\to d$ for some $c,d\in\R$ and some $p\in\Z[t]$; \text{ or }

 \item $|a(t)-t/m|\ll \log{t}$ for some $m\in\Z$.
\end{itemize}
\end{theorem}
\begin{remarks}
$\bullet$ The necessity of these conditions can be  seen using
rotations on the circle (see \cite{BKQW}).
In the  case were $X=\T$ their sufficiency was established in \cite{BKQW}.


$\bullet$  Unlike the case of integer polynomial sequences, if $p\in \R[t]$, then
the sequence  $(b^{[p(n)]}x)_{n\in\N}$ may not be equidistributed in a finite union of sub-nilmanifolds  of $X$. For example, when $X=\T (=\R/\Z)$,
the sequence $(-[n\sqrt{2}]/\sqrt{2}\, \Z)_{n\in\N}$
is equidistributed
in  the set $\big\{ t\Z\colon \{t\}\in [0,1/\sqrt{2}]\big\}$.
\end{remarks}
It seems sensible to assert that the first condition in Theorem~\ref{T:A} is satisfied by the ``typical'' function in $\H$ with polynomial growth. It turns out that in this ``typical'' case,
restricting the  parameter $n$ of a  nil-orbit $(b^n\Gamma)_{n\in\N}$  to the range of  the sequence
$([a(n)])_{n\in\N}$, does not change its limiting distribution:
\begin{theorem}\label{T:B}
Let $a\in\H$ have  polynomial growth and satisfy   $|a(t)-cp(t)|\succ \log t$ for every $c\in \R$ and
$p\in \Z[t]$.

$(i)$ If     $X=G/\Gamma$ is a nilmanifold, with $G$ connected and simply connected, then   for every
$b\in G$ and $x\in X$
the sequence $(b^{a(n)}x)_{n\in\N}$ is equidistributed in the nilmanifold
 $ \overline{(b^sx)}_{s\in \R}$.

$(ii)$ If   $X=G/\Gamma$ is a nilmanifold, then for every  $b\in G$ and   $x\in X$ the
  sequence $(b^{[a(n)]}x)_{n\in\N}$ is equidistributed in the nilmanifold $ \overline{(b^nx)}_{n\in \N}$.
 \end{theorem}
\begin{remark}
Suppose that  we  want the  conclusion $(i)$ (or $(ii)$)  to be true only for some
fixed  $b\in G$. Then our proof shows that the assumption can be relaxed to the following:
$|a(t)-cp(t)|\succ \log t$ for every $p\in \Z[t]$,  and
every $c\in \R$ of the form $q/\beta$ where $q\in \Q$ and $\beta$ is some non-zero eigenvalue for
the nilrotation by $b$ (this means $f(bx)=e(\beta)f(x)$ \text{ for some } non-constant $f\in L^2(m_X)$).
 A special case of this stronger result  (take $G=\R$, $\Gamma=\Z$,  and $b=1$)
  gives one of the main  results in \cite{Bo2}, stating that if $a\in \H$ has polynomial
  growth and satisfies $|a(t)-p(t)|\succ \log{t}$  for every $p\in\Q[t]$, then the sequence $(a(n)\Z)_{n\in\N}$ is equidistributed in $\T$.
\end{remark}

 \subsubsection{Several nil-orbits}
 We give an equidistribution result involving nil-orbits of several Hardy sequences.
 We say that the functions $a_1(t),\ldots,a_\ell(t)$  have \emph{different growth rates} if the quotient of any two of these functions  converges  to $\pm \infty$ or to $0$.

\begin{theorem}\label{T:several}
Suppose that the functions  $a_1(t),\ldots,a_{\ell}(t)$ belong to the same Hardy field,  have different growth rates, and satisfy
 $t^{k_i}\log{t}\prec a_i(t)\prec t^{k_i+1} $ for some  $ak_i\in \N$.

$(i)$ If   $X_i=G_i/\Gamma_i$ are  nilmanifolds, with $G_i$ connected and simply connected,
then for every  $b_i\in G_i$ and $x_i\in X_i$,
the sequence $$
({b}^{a_1(n)}_1x_1,\ldots,{b}^{a_{\ell}(n)}_{\ell}x_{\ell})_{n\in\N}
$$ is equidistributed in the nilmanifold
 $
 \overline{(b^s_1x_1)}_{s\in \R}\times \cdots\times \overline{(b^s_{\ell}x_\ell)}_{s\in \R}.
 $

$(ii)$ If  $X_i=G_i/\Gamma_i$ are  nilmanifolds, then   for every  $b_i\in G_i$, and
 $x_i\in X_i$,
the sequence $$
({b}^{[a_1(n)]}_1x_1,\ldots,{b}^{[a_{\ell}(n)]}_{\ell}x_{\ell})_{n\in\N}
$$ is equidistributed in the nilmanifold
$
 \overline{(b^n_1x_1)}_{n\in \N}\times \cdots\times \overline{({b}^n_{\ell}x_{\ell})}_{n\in \N}.
 $
\end{theorem}
\begin{remark}

The preceding  result contrasts the case of polynomial sequences, where different growth does not imply
 simultaneous equidistribution for the corresponding nil-orbits. For example,
  there exists a connected nilmanifold $X=G/\Gamma$ and an  ergodic element $b\in G$,
 such that the sequence $(b^n\Gamma,b^{n^2}\Gamma)_{n\in\N}$ is not even dense in $X\times X$
(see \cite{FK}). On the other hand, our result shows that if for instance $a(t)=t^{\sqrt{2}}$, then for every nilmanifold $X$ and ergodic element $b\in G$ the sequence   $(b^{[a(n)]}\Gamma,b^{[(a(n))^2]}\Gamma)_{n\in\N}$
 is equidistributed in $X\times X$.
\end{remark}

It may very well be  the case that the hypothesis of Theorem~\ref{T:several} can be relaxed to give a much stronger result.
The following is a closely related conjecture:
\begin{conjecture}\label{C:B'}
Let  $a_1(t),\ldots,a_{\ell}(t)$ be functions that belong to the same Hardy field and have polynomial growth.
Suppose further that every non-trivial linear combination
$a(t)$ of these functions satisfies $|a(t)-cp(t)|\succ \log{t}$ for every $c\in \R$ and $p\in \Z[t]$.

Then for every nilmanifold $X=G/\Gamma$, $b_i\in G$, and
 $x_i\in X$,
the sequence $$
({b}^{[a_1(n)]}_1x_1,\ldots,{b}^{[a_{\ell}(n)]}_{\ell}x_{\ell})_{n\in\N}
$$ is equidistributed in the nilmanifold
$
 \overline{(b^n_1x_1)}_{n\in \N}\times \cdots\times \overline{({b}^n_{\ell}x_{\ell})}_{n\in \N}.
 $
\end{conjecture}
\subsubsection{More general classes of functions}
We make some remarks about the extend of the functions our methods cover that do not necessarily belong to some Hardy field.

 The conclusions of Theorem~\ref{T:B} hold if for some $k\in \N$ the function
  $a(t)$ is $(k+1)$-times differentiable for large $t\in \R$  and satisfies:
\begin{equation}\label{E:tempered}
    a^{(k+1)}(t)\to 0 \text{  monotonically} ,\  \text{ and } \  \ t|a^{(k+1)}(t)|\to \infty.
 \end{equation}
 (If $a\in \H$, then \eqref{E:tempered} is equivalent to  ``$t^k\log{t}\prec a(t)\prec t^{k+1}$''.)
 To see this, one can repeat the proof  of Theorem~\ref{T:several}  in this particular setup.
 More generally, the conclusion of Theorem~\ref{T:B} holds for
functions $a(t)$  that satisfy the following less restrictive
conditions: For some
$k\in \N$ the function $a\in C^{k+1}(\R_+)$ satisfies
\begin{equation}\label{E:complicated}
|a^{(k+1)}(t)| \text{ decreases to zero}, \quad 1/t^k\prec a^{(k)}(t)\prec 1,
\quad  \text{and } \quad  (a^{(k+1)}(t))^k\prec (a^{(k)}(t))^{k+1}.
\end{equation}
(If $a\in \H$ , then \eqref{E:complicated} is equivalent to
``$a(t)$ has polynomial growth and $|a(t)-p(t)|\succ \log{t}$ for every $p\in \R[t]$''.)
One can see this by  repeating verbatim part of the proof of Theorem~\ref{T:B}. The reader is advised to think
of the second condition in \eqref{E:complicated}
as the most important one and the other two as  technical necessities (for functions in $\H$ the second condition implies the other two).

Theorem~\ref{T:several}   can be proved  for functions $a_i(t)$ that satisfy condition \eqref{E:tempered}
 for some $k_i\in \N$ (call this integer the type of $a_i(t)$), and also, for every $k\in\N$ every non-trivial
 linear combination  of those functions $a_i(t)$ that have type $k$ also satisfies \eqref{E:tempered}.

 As for  Theorem~\ref{T:A}, unless one works within
 a ``regular" class of functions  like $\H$, it seems hopeless to state a result
 with explicit necessary and sufficient conditions.




\subsubsection{Random  sequences of sub-exponential growth}
So far, we have given examples of   sequences that are pointwise good for nilsystems and
have polynomial growth.
For Hardy sequences of super-polynomial growth,  it is indicated
in \cite{Bo2} that no growth condition should suffice to guarantee equidistribution on $\T$.
On the other hand, explicit  sequences of super-polynomial growth
like $(e^{\sqrt{n}})_{n\in\N}$ or  $(e^{(\log{n})^2})_{n\in\N}$ are expected
to be pointwise good for nilsystems, but proving this seems to be out of reach at the moment, even for rotations
on $\T$.

Nevertheless, using a probabilistic argument, we shall show that there exist
 very sparsely distributed sequences that are pointwise good for nilsystems.
In fact, loosely speaking,  we shall see that the only growth condition prohibiting the existence of such examples is exponential growth.

 Our probabilistic setup is as follows. Let
   $(\sigma_n)_{n\in\N}$ be a decreasing sequence of reals in $[0,1]$.
  We shall construct random sets of integers by including each integer $n$ in the set with probability $\sigma_n\in [0,1]$.
   More formally, let  $(\Omega, \Sigma, P)$ be a probability space, and
 $(X_n)_{n\in\N}$ be a sequence of $0-1$ valued independent random variables with
$P(\{\omega\in \Omega\colon X_n(\omega)=1\})=\sigma_n$. Given $\omega\in \Omega$ we construct the set of positive
integers  $A^\omega$  by taking $n\in A^\omega$ if and only if  $X_n(\omega)=1$. By writing the elements of $A^\omega$
in increasing order we get a sequence $(a_n(\omega))_{n\in\N}$.

 If $\sigma_n=1/n^c$
  where $c\in [0,1)$, then the resulting random sequence has almost surely polynomial growth (in fact it is  asymptotic to $n^{1/(1-c)}$).
    If $\sigma_n=1/n$, then almost surely the  resulting random sequence is bad for pointwise convergence results
      even for circle rotations (see \cite{JLW}).
  Therefore,  it makes sense to restrict our attention to the case where $\sigma_n\succ 1/n$.
  By choosing $\sigma_n$ appropriately, we can get
  examples of random sequences with any prescribed sub-exponential growth.



In  \cite{Bo1} (and subsequently in \cite{Bo})  it was shown that if $\lim_{n\to\infty}  n\sigma_n=\infty$, then
almost surely, the random sequence $(a_n(\omega))_{n\in\N}$ is pointwise good for convergence of
 rotations on the circle.
 We extend this result to rotations on nilmanifolds by showing the following:
\begin{theorem}\label{T:random}
Let $(\sigma_n)_{n\in\N}$ be a decreasing sequence  of reals  satisfying
$\lim_{n\to\infty}  n\sigma_n=\infty$.

Then almost surely, the random sequence $(a_n(\omega))_{n\in\N}$ is pointwise good
 for nilsystems.
\end{theorem}
\begin{remark} As it will become clear from the proof,
 the condition $\lim_{n\to\infty}  n\sigma_n=\infty$ can be replaced with the condition
 $\lim_{N\to \infty}
\frac{\sum_{1\leq n\leq N}\sigma_n}{\log{N}}=\infty.
$
Furthermore, our method of proof will show that almost surely, the limits
$\lim_{N\to\infty}\frac{1}{N}\sum_{n=1}^NF(b^{a_n(\omega)}x)$ and $\lim_{N\to\infty}\frac{1}{N}\sum_{n=1}^NF(b^nx)$
are equal for every nilmanifold $X=G/\Gamma$, $F\in C(X)$,  $b\in G$, and $x\in X$.
\end{remark}
\subsection{Applications}
We give some rather straightforward applications of the preceding equidistribution results.
We only sketch their proofs  leaving some routine details to the reader. For aesthetic reasons,
we represent elements $t\Z$ of $\T=\R/\Z$ by  $t$.

The first is  an equidistribution result on  $\T$, which we do
not see how to handle using conventional exponential sum techniques for   $k\geq 2$.
\begin{theorem}\label{T:Application1}
Let $a\in\H$ have   polynomial  growth
and satisfy   $|a(t)-cp(t)|\succ \log t$ for every $c\in \R$ and $p\in \Z[t]$.

Then for  every  $k\in\N$ and irrational $\beta$, the sequence $([a(n)]^k\beta)_{n\in\N}$ is equidistributed in $\T$.
\end{theorem}
\begin{remark}
A standard modification of our argument gives the following more general conclusion: for every $q\in \Z[t]$ non-constant and irrational $\beta$, the sequence $\big(q([a(n)])\beta\big)_{n\in\N}$ is equidistributed in $\T$. A similar extension holds
for Theorems~\ref{T:Application3}, \ref{T:Application4}, and Theorem~\ref{T:Application2} (with $\ell$ non-constant polynomials).
\end{remark}
\begin{proof}[Proof (Sketch)]
Suppose for convenience that $k=2$.
We define the transformation
$T\colon\mathbb{T}^2\to\mathbb{T}^2$ by
$ T(x,y)=\big(x+\beta,y+2x+\beta\big).$
It is well known that the resulting system is isomorphic to a nilsystem (and the conjugacy map is continuous),
and that this system is ergodic if $\beta$ is irrational.
Applying  Theorem~\ref{T:B} we  get that the  sequence $(T^{[a(n)]}(0,0))_{n\in\N}$ is equidistributed in $\T^2$.
An easy computation shows that $T^n(0,0)=(n\beta,n^2\beta)$, therefore for every $F\in C(\T^2)$ we have
$$
\lim_{N\to \infty} \frac{1}{N}\sum_{n=1}^N F([a(n)]\beta,[a(n)]^2\beta)=\int F \ dm_{\T^2},
$$
where $m_{\T^2}$ denotes the normalized Haar measure on $\T^2$.
Using this identity for $F(x,y)=e(ky)$, where $k$ is a non-zero integer, we get
$$
\lim_{N\to \infty}  \frac{1}{N} \sum_{n=1}^N e(k[a(n)]^2\beta )=0.
$$
This shows that the sequence $([a(n)]^2\beta)_{n\in\N}$ is equidistributed in $\T$.
\end{proof}
 Similarly, we can deduce from Theorem~\ref{T:several} the following result:
\begin{theorem}\label{T:Application2}
Let   $a_1(t),\ldots,a_{\ell}(t)$ be functions that belong to the same Hardy field, have different growth rates, and satisfy
 $t^k\log{t}\prec a_i(t)\prec t^{k+1} $ for some  $k=k_i\in \N$.

Then for every  $l_i\in\N$ and irrationals $\beta_i$, the sequence
$$([a_1(n)]^{l_1}\beta_1,\ldots,[a_{\ell}(n)]^{l_{\ell}}\beta_{\ell} )_{n\in\N}$$ is equidistributed in $\T^{\ell}$.
\end{theorem}
 Next  we give an application to
  ergodic theory. We say that a sequence of integers $(a(n))_{n\in\N}$  is \emph{good for mean convergence}, if
for every invertible measure preserving system  $(X,\mathcal{B},\mu,T)$ and $f\in L^2(\mu)$ the averages
$
 \frac{1}{N}\sum_{n=1}^N f(T^{a(n)}x)
$
converge  in $L^2(\mu)$ as $N\to \infty$. Using the spectral theorem for unitary operators,
one can see that a sequence  $(a(n))_{n\in\N}$  is \emph{good for mean convergence}
if and only if for every $t\in \R$ the sequence $(a(n)t)_{n\in\N}$ has a limiting distribution.

\begin{theorem}\label{T:Application3}
Let $a\in \H$ have polynomial growth and  $k\in\N$.

Then the sequence  $([a(n)]^k)_{n\in\N}$
is good for mean convergence  if and only if one of the three conditions in Theorem~\ref{T:A} is satisfied.
\end{theorem}
\begin{remark}
For $k=1$ this result was established in \cite{BKQW}.
\end{remark}
\begin{proof}[Proof (Sketch)]
The necessity of the conditions can be seen exactly as in the proof of Theorem~\ref{T:Bos}.
To prove the sufficiency, we apply  Theorem~\ref{T:A} for some appropriate unipotent affine transformations of  some finite dimensional tori. We deduce that for every $k\in \N$ and $t\in\R$ the sequence
 $([a(n)]^k t)_{n\in\N}$ has a limiting distribution. As explained before, this implies   that the sequence $([a(n)]^k)_{n\in\N}$ is good for mean convergence.
\end{proof}
Lastly, we give a recurrence result for measure preserving systems, and a corresponding combinatorial consequence. We say that a sequence of integers   $(a(n))_{n\in\N}$  is \emph{good for recurrence}, if
for every invertible measure preserving system  $(X,\mathcal{B},\mu,T)$  and set $A\in \mathcal{B}$
with $\mu(A)>0$, one has $\mu(A\cap T^{-a(n)}A)>0$ for some $n\in \N$ such that $a(n)\neq 0$.
Using the correspondence principle of Furstenberg (\cite{Fu1}) one can see that this notion is equivalent to the following one: A sequence of integers $(a(n))_{n\in\N}$ is \emph{intersective} if every set of integers $\Lambda$ with positive upper density contains two distinct
elements $x$ and $y$ such that $x-y=a(n)$ for some $n\in\N$.
\begin{theorem}\label{T:Application4}
Let $a\in \H$ have  polynomial growth   and satisfy $|a(t)-cp(t)|\succ \log{t}$
for every $c\in \R$ and $p\in \Z[t]$.

Then for every $k\in\N$ the sequence  $([a(n)]^k)_{n\in\N}$
is good for recurrence (or intersective).
\end{theorem}
\begin{remarks}
$\bullet$ For $k=1$ this  result can be  deduced from the equidistribution results in \cite{Bo2}.

$\bullet$ A  more tedious argument can be used to show that the following weaker assumption suffices:
``$a\in \H$  has polynomial growth and satisfies  $|a(t)-cp(t)|\to \infty$ for every $c\in \R$ and
$p\in \Z[t]$". (By combining  the spectral theorem and an argument similar to one used in the proof of  Proposition~6.5
in \cite{Fr3}, one can handle the case where  and $|a(t)-cp(t)|\ll \log{t}$ for some $c\in \R$ and $p\in \Z[t]$.)
\end{remarks}
\begin{proof}[Proof (Sketch)]
We apply Theorem~\ref{T:B} for some appropriate unipotent affine transformations of  finite dimensional tori. We deduce that for every $k\in \N$ and $t\in\R$ the sequence
 $([a(n)]^k t)_{n\in\N}$ has the same limiting distribution as the sequence
  $(n^kt)_{n\in\N}$. Using this and the spectral theorem for unitary operators,
  we conclude that for every invertible measure preserving system  $(X,\mathcal{B},\mu,T)$  and set $A\in \mathcal{B}$ , we have
$$
\lim_{N\to \infty} \frac{1}{N}\sum_{n=1}^N\mu(A\cap T^{-[a(n)]^k}A)=\lim_{N\to \infty} \frac{1}{N}\sum_{n=1}^N\mu(A\cap T^{-n^k}A).
$$
Since the last limit is known to be positive  whenever $\mu(A)>0$ (\cite{Fu1}), the previous identity shows that the sequence $[a(n)]^k$ is good for recurrence.
\end{proof}

More delicate applications of the equidistribution results presented in Section~\ref{SS:equidistribution} include
statements about multiple recurrence and convergence of multiple  ergodic averages, and related combinatorial consequences.
 Such results require
much   extra work and will be presented in a forthcoming paper (\cite{Fr3}).







\subsection{Structure of the article} In Section~\ref{S:background} we give the necessary background on Hardy fields
and state some equidistribution results on nilmanifolds that will be used later.

In Section~\ref{S:warmup} we work on a model equidistribution problem that  helps us illustrate some of the ideas
needed to prove Theorems~\ref{T:A} and \ref{T:B}. We give a new proof of a result of Boshernitzan on equidistribution
of the fractional parts of Hardy sequences of polynomial growth.

In Section~\ref{S:single} we prove Theorems~\ref{T:A} and \ref{T:B}. The key ingredients are: (i) a reduction step that enables
us to ``remove'' the  integer parts and deal with equidistribution properties on nilmanifolds $X=G/\Gamma$ with $G$ connected and
simply connected, (ii) the proof technique of the model problem described in Section~\ref{S:warmup}, and
 (iii) some quantitative equidistribution results of Green and Tao.

  In Section~\ref{S:multiple} we prove Theorem~\ref{T:several}. The proof strategy is similar with that of
   Theorems~\ref{T:B}, with the exception of a key technical difference  that is illustrated
   using a  model equidistribution problem.

  In Section~\ref{S:random} we prove Theorem~\ref{T:random}. We adapt an argument of Bourgain
  that worked for circle rotations to our more complicated non-Abelian setup.

\subsection{Notational conventions.} The following notation will be
used throughout the article: $\N=\{1,2,\ldots\}$, $\T^k=\R^k/\Z^k$, $Tf=f\circ T$,
$e(t)=e^{2\pi i t}$, $[t]$ denotes the integer part of $t$,  $\{t\}=t-[t]$, $\norm{t}=d(t,\Z)$,
$\E_{ n\in A} a(n)=\frac{1}{|A|}\sum_{ n\in A} a(n)$.
By $a(t)\prec b(t)$ we mean
$\lim_{t\to\infty}a(t)/b(t)=0$,    by $a(t)\sim b(t)$ we mean $\lim_{t\to\infty}a(t)/b(t)$ is a non-zero real number, and by $a(t)\ll b(t)$ we mean $|a(t)|\leq C |b(t)|$
for some constant $C$ for  all large $t$.
We use the symbol $\ll_{w_1,\ldots, w_k}$  when
 some expression is  majorized by some  other expression and  the implied constant
 depends on the parameters $w_1,\ldots, w_k$. By  $o_{N\to\infty;w_1,\ldots,w_k}(1)$  we denote a quantity that goes to zero when the parameters $w_1,\ldots,w_k$ are fixed and
$N\to\infty$ (when there is no danger of confusion we may omit the parameters).
We often  write $\infty$ instead of $+\infty$.  For aesthetic reasons,
we sometimes represent elements $t\Z$ of $\T=\R/\Z$ by  $t$.

\bigskip

\noindent {\bf Acknowledgement.}  We thank the referee for providing constructive comments.

\section{Background on Hardy fields and nilmanifolds}\label{S:background}
\subsection{Hardy fields}\label{SS:Hardy}
  Let $B$ be the collection of equivalence classes of real valued
  functions  defined on some half line $(c,\infty)$, where we
  identify two functions if they agree eventually.\footnote{The
    equivalence classes just defined are often called \emph{germs of
      functions}. We choose to use the word function when we refer to
    elements of $B$ instead, with the understanding that all the
    operations defined and statements made for elements of $B$ are
    considered only for sufficiently large values of $t\in \R$.}  A
  \emph{Hardy field} is a subfield of the ring $(B,+,\cdot)$ that is
  closed under differentiation. With $\H$ we denote the \emph{union of all
  Hardy fields}. If $a\in \H$ is defined in $[1,\infty)$ (one can always
  choose such a representative of $a(t)$)
  we call the sequence $(a(n))_{n\in\N}$ a \emph{Hardy sequence}.

A particular example of a Hardy field is the set of all rational functions with real coefficients.
 Another example is the set  $\LE$  that consists of
all  \emph{logarithmic-exponential functions} (\cite{Ha10},
\cite{Ha1}), meaning  all functions defined on some  half line
$(c,\infty)$ by a finite combination of the symbols $+,-,\times, :,
\log, \exp$, operating on the real variable $t$ and on real
constants.
 For example the functions $t^{\sqrt{2}}$, $t^2+t\sqrt{2}$,  $t\log{t}$, $e^{t^2}$,
$e^{\sqrt{\log \log t}}/\log(t^2+1)$,  are all elements of $\LE$.

We
collect here some properties that illustrate the richness of $\H$.
More information about   Hardy fields can be found in
the paper \cite{Bo2} and the references therein.

$\bullet$ $\H$ contains the set $\LE$  and  anti-derivatives of elements of $\LE$.

$\bullet$ $\H$ contains several other functions not in $\LE$, like the
functions $\Gamma(t)$, $\zeta(t)$,  $\sin{(1/t)}$.

$\bullet$ If $a\in \LE$ and $b\in\H$, then  there exists a Hardy field
containing both $a(t)$ and $b(t)$.

$\bullet$ If $a\in\LE$, $b\in \H$, and $b(t)\to \infty$, then $a\circ
  b\in \H$.

 \ \ \  If $a\in\LE$, $b\in \H$, and $a(t)\to \infty$, then $b\circ a\in \H$.

$\bullet$ If $a$ is a continuous function that is algebraic over some
  Hardy field, then $a\in\H$.

Using these properties it is easy to check that, for example, the sequences
$(\log \Gamma(n^2))_{n\in\N}$, $(n^{\sqrt{5}}\zeta(n))_{n\in\N}$,
 $((\text{Li}(n))^2)_{n\in\N}$ ($\text{Li}(t)=\int_2^t 1/\ln{s} \ ds$),  and
 the sequences that appear inside the integer parts
in \eqref{E:examples}, are Hardy sequences with
polynomial growth.
On the other hand, sequences that oscillate, like $(\sin{n})_{n\in\N}$,
 $(n\sin{n})_{n\in\N}$, or the sequence $(e^n+\sin{n})_{n\in\N}$ are not Hardy  sequences.

We mention some basic properties of elements of $\H$ relevant to our study. Every element of $\H$ has eventually constant sign
(since it has a multiplicative inverse).
 Therefore, if  $a\in \H$, then $a(t)$ is eventually monotone (since $a'(t)$ has eventually constant sign), and  the limit
$\lim_{t\to \infty} a(t)$ exists (possibly infinite).  Since  for every two
functions $a\in \H, b\in \LE$ $(b\neq 0)$, we have $a/b\in \H$, it follows that   the asymptotic growth ratio
$\lim_{t\to \infty}a(t)/b(t)$ exists (possibly infinite). This last property is key,
since it will often justify our use of  l'Hopital's rule.
\emph{We are going to freely use all these properties without any further explanation in the sequel.}

We caution the
 reader  that although every function in $\H$ is
asymptotically comparable with every function in $\LE$,
  some functions in $\H$ are not comparable. This defect of
 $\H$ will only play a role in one of our results (Theorem~\ref{T:several}), and can be sidestepped  by  restricting
 our attention to functions that belong to the same Hardy field.

A key property of elements of $\H$ with
polynomial growth is that we can relate their growth rates with
the growth rates of their derivatives:
\begin{lemma}\label{L:properties}
Suppose that   $a\in \H$ has polynomial growth.   We have the following:

$(i)$ If $t^\varepsilon \prec a(t)$ for some $\varepsilon>0$, then $a'(t)\sim a(t)/t$.

$(ii)$ If $t^{-k}\prec a(t)$ for some $k\in\N$, and $a(t)$ does not converge to a non-zero constant,
then $a(t)/(t(\log{t})^2)\prec a'(t)\ll a(t)/t$.
\end{lemma}
\begin{remark} The assumption of polynomial growth is essential, to see this take
$a(t)=e^t$. To see that the other assumptions in parts  $(i)$ and $(ii)$
are essential take $a(t)=\log{t}$ for part $(i)$, and
  $a(t)=e^{-t}$,  $a(t)=1+1/t$ for part $(ii)$.
\end{remark}
\begin{proof}
First we deal with part $(i)$. Applying l'Hopital's rule  we get
\begin{equation}\label{E:log}
\lim_{t\to\infty}\frac{ta'(t)}{a(t)}=\lim_{t\to\infty}\frac{(\log{|a(t)|})'}{(\log{t})'}=
\lim_{t\to\infty}\frac{\log{|a(t)|}}{\log{t}}.
\end{equation}
Since $t^\varepsilon \prec a(t)$ for some $\varepsilon>0$ and $a(t)$ has polynomial growth,   the last  limit is
a positive real number. Hence, $a'(t)\sim a(t)/t$, proving part $(i)$.

Next we deal with part $(ii)$.
First notice that since   $a(t)$ does not converge to a non-zero constant  we can assume that
 either  $|a(t)|\to \infty$  or $|a(t)|\to 0$ .

 We show that $a'(t)\ll a(t)/t$.  Since $\lim_{t\to \infty} \log{|a(t)|}=\pm \infty$
 we can apply  l'Hopital's rule to get \eqref{E:log}.
Since  $a(t)$ has polynomial growth and $t^{-k} \prec a(t)$ for some $k\in\N$,
we have that  the limit  $\lim_{t\to\infty} \log{|a(t)|}/\log{t}$ is finite.
 Using \eqref{E:log} we conclude that  the same holds for the limit
$\lim_{t\to\infty} (ta'(t))/a(t)$. It follows that  $a'(t)\ll a(t)/t$.

Finally we show that  $a(t)/(t(\log{t})^2)\prec a'(t)$. Equivalently,  it suffices to show that the  limit
$$
\lim_{t\to\infty}\frac{t(\log{t})^2a'(t)}{a(t)}
$$
is infinite.
Arguing by contradiction, suppose  this is not the case. Then
$$
(\log{|a(t)|})'\ll \frac{1}{t(\log{t})^2},
$$
and integrating we get
$$
\log{|a(t)|}\ll \frac{1}{\log{t}}+c
$$ for some $c\in \R$.
It follows   that $\log{|a(t)|}$ is bounded, which contradicts the fact that $|a(t)|\to \infty$ or $0$.
This completes the proof.
\end{proof}

Following \cite{Bo2}, for  a non-negative integer $k$  we say that:
\begin{itemize}
\item[(i)] The function $a\in \H$  has \emph{type $k$} if $a(t)\sim t^k$.

\item[(ii)]  The function $a\in \H$  has \emph{type $k^+$} if $t^k\prec a(t)\prec t^{k+1}$.
\end{itemize}
It is easy to show the following:
\begin{lemma}[Boshernitzan~\cite{Bo2}]\label{L:type}
Suppose that $a\in \H$ has polynomial growth. Then

$(i)$ There exists a non-negative integer $k$ such that $a(t)$ has type  either $k$ or $k^+$.

$(ii)$ If $a(t)$ has type $k$, then $a(t)=ct^k+b(t)$ for some non-zero $c\in \R$ and $b\in \H$ with $b(t)\prec t^k$.
\end{lemma}
Applying Lemma~\ref{L:properties} repeatedly we get:
\begin{corollary}\label{C:properties}
  Suppose that $a\in \H$ has type  $k^+$ for some non-negative integer $k$.

   Then for every $l\in \N$ with  $l\leq k$ we have $a^{(l)}(t)\sim a(t)/t^l$,
and for every $l\in \N$ we have 
 $$
 a(t)/(t^l(\log{t})^2)\prec a^{(l)}(t)\ll a(t)/t^l.
 $$
\end{corollary}
\begin{remark}
The conclusion fails for some functions of  type $k$. Indeed, if $a(t)=1+1/t$, then
$a'(t)\prec a(t)/(t(\log{t})^2)$.
\end{remark}

  \subsection{Nilmanifolds}\label{SS:nil}
Fundamental properties of rotations on  nilmanifolds, related to our study,
were studied in \cite{AGH}, \cite{Pa}, \cite{Les}, and \cite{L2}.
Below we summarize some facts that we shall  use, all the proofs
can be found or deduced from \cite{L2} and \cite{CG}.

 Given a topological group $G$, we denote its identity element by $\text{id}_G$.
By $G_0$ we denote the connected component of $\text{id}_G$.  If $A,
B\subset G$, then $[A,B]$ is defined to be the subgroup generated by
elements of the form $\{[a,b]:a\in A, b\in B\}$ where $[a,b]=ab
a^{-1}b^{-1}$. We define the commutator subgroups recursively by
$G_1=G$ and $G_{k+1}=[G, G_{k}]$. A group $G$ is said to be {\it
   nilpotent} if  $G_{k}=\{\text{id}_G\}$ for some $k\in\N$.
If $G$ is a  nilpotent Lie group and $\gG$ is a discrete
cocompact subgroup, then the compact homogeneous space $X = G/\gG$ is called a
{\it nilmanifold}.  The group $G$ acts on $G/\gG$ by left
translation where the translation by a fixed element $b\in G$ is given
by $T_{b}(g\gG) = (bg) \gG$. We denote by  $m_X$ the normalized \emph{Haar measure} on $X$, meaning,
 the unique probability measure that is
 invariant under the action of $G$ by left
translations and is defined on the Borel $\sigma$-algebra of $X$.
We
call the elements of $G$ \emph{nilrotations}. A nilrotation $b\in G$
 \emph{acts ergodically on $X$},  if the sequence $(b^n\Gamma)_{n\in\N}$ is dense
in $X$.  When the nilmanifold $X$ is implicit we shall often simply
say that  a nilrotation $b\in G$ is \emph{ergodic}.
It can be shown that if $b\in G$ is ergodic, then  for
every $x\in X$ the sequence $(b^nx)_{n\in\N}$ is equidistributed in $X$.
A nilrotation $b\in G$ is  \emph{totally  ergodic}, if for every $r\in \N$ the  nilrotation $b^r$
is  ergodic. If the nilmanifold $X$ is connected it can be shown that every ergodic nilrotation
is in fact totally ergodic.

\begin{example}[Heisenberg nilmanifold]\label{Ex1}
  Let $G$ be the nilpotent  group that consists of all upper triangular matrices of the form
  $\left(\begin{smallmatrix} 1& x& z\\ 0& 1& y\\0& 0 & 1\end{smallmatrix}\right)$ with real entries.
   If we only allow integer entries we get a subgroup  $\Gamma$
  of $G$ that is discrete and cocompact. Then $G/\Gamma$ is a nilmanifold. It can be shown that a nilrotation
   $b=\left(\begin{smallmatrix} 1& \alpha& \gamma\\ 0& 1& \beta\\0& 0 & 1\end{smallmatrix}\right)$ is ergodic
   if and only if the numbers $1$, $\alpha$, and $\beta$ are rationally independent.
  \end{example}
Let $G$ be a connected and simply connected Lie group and
  $\exp\colon \mathfrak{g}\to G$ be the exponential map, where $\mathfrak{g}$ is the
  Lie algebra of $G$. Since $G$ is a connected and simply connected nilpotent Lie group, it is well known
   that the exponential
  map is a bijection.  For   $b\in G$ and $s\in\R$ we
define the element $b^s$ of $G$  as follows:
If  $X\in \mathfrak{g}$ is such that
  $\exp(X)=b$, then   $b^s=\exp(sX)$.

A more intuitive way to  make sense of the element $b^s$ is by thinking of $G$ as a matrix group;
then  $b^s$ is  the element one gets after replacing $n$ by $s$
in the formula giving the elements of the matrix $b^n$.
It is instructive to compare the two equivalent ways of defining $b^s$ in the following example.
  \begin{example}[Heisenberg nilflow]
 Let $X$ be the Heisenberg nilmanifold. Then the  exponential map is given by
 $\exp \left(\begin{smallmatrix} 0 & x & z\\ 0& 0& y\\0& 0& 0\end{smallmatrix}\right)=
 \left(\begin{smallmatrix} 1& x& z+\frac{1}{2}xy\\ 0& 1& y\\0& 0 & 1\end{smallmatrix}\right) $.
  As a consequence, if $b=\left(\begin{smallmatrix}1 & \alpha & \gamma \\ 0& 1& \beta\\0& 0& 1\end{smallmatrix}\right)$, then $\exp \left(\begin{smallmatrix}0 & \alpha & \gamma-\frac{1}{2}\alpha\beta \\ 0& 0& \beta\\0& 0& 0\end{smallmatrix}\right)=b$, and  a short computation shows that
 $b^s=\left(\begin{smallmatrix} 1& s\alpha& s\gamma+\frac{s(s-1)}{2}\alpha \beta \\ 0& 1& s\beta\\0& 0& 1\end{smallmatrix}\right)$.
Alternatively,  one can  find the same formula for $b^s$ after replacing $n$ by $s$ in the formula
 $b^n=\left(\begin{smallmatrix} 1& n\alpha& n\gamma+\frac{n(n-1)}{2}\alpha \beta \\ 0& 1& n\beta\\0& 0& 1\end{smallmatrix}\right)$.
  \end{example}

Next we record some  basic facts that we will frequently use:

(\emph{Basic properties of $b^s$}). If $G$ is a connected and simply connected Lie group, then
for $b\in G$ the map  $s\to b^s$ is  continuous, and for $s,s_1,s_2\in \R$
one has the identities
 $b^{s_1+s_2}=b^{s_1}\cdot b^{s_2}$,
  $(b^{s_1})^{s_2}=b^{s_1s_2}$, and  $(gbg^{-1})^s=gb^sg^{-1}$.

(\emph{Ratner's theorem, nilpotent case}) Let  $X=G/\Gamma$ be a nilmanifold. Then for every $b\in G$ the  set $X_b=\overline{\{b^n\Gamma, n\in\N\}}$ has
 the form $H/\Delta$, where $H$ is a closed subgroup of $G$ that contains $b$, and $\Delta=H\cap \Gamma$
is a discrete cocompact subgroup of $H$. Furthermore, the sequence $(b^n\Gamma)_{n\in\N}$ is equidistributed in $X_b$.

Likewise, if $G$ is
connected and simply connected, and $b\in G$, let  $Y_b=\overline{\{b^s\Gamma, s\in\R\}}$. Then $Y_b$
has the form $H/\Delta$ where $H$ is a closed connected and simply connected subgroup of $G$
that contains all elements $b^s$ for $s\in \R$, and $\Delta$ is a discrete cocompact subgroup of $H$. Furthermore, the nilflow
$(b^s\Gamma)_{s\in\R}$ is equidistributed in $Y_b$.

  (\emph{Change of base point formula}). Let  $X=G/\Gamma$ be a nilmanifold. As mentioned before,
  for every $b\in G$ the  nil-orbit  $(b^n\Gamma)_{n\in\N}$ is equidistributed in
  the set $X_b=\overline{\{b^n\Gamma, n\in\N\}}$. Using the identity
$b^{n}g=g(g^{-1}bg)^{n}$ we see that the nil-orbit $(b^ng\Gamma)_{n\in\N}$
is equidistributed in the set $g\cdot X_{g^{-1}bg}$. A similar formula holds
when $G$ is connected and simply connected and we replace the integer parameter $n$ with the real parameter $s$
and  the nilmanifold $X_{b}$ with $Y_b$.

(\emph{Lifting argument}).  In several instances it will
be convenient for us to assume that a nilmanifold
 $X$ has a representation $G/\Gamma$ with  $G$  connected and simply connected.
 To get this extra assumption we argue as follows (see \cite{L2}):
 Since all our results deal with the
 action on $X$ of  finitely many elements of $G$
 we conclude
  that for the purposes of this paper, we can, and will always assume that the discrete group $G/G_0$ is finitely generated.
In this case, one can show  that $X=G/\Gamma$ is isomorphic to a
 sub-nilmanifold of a nilmanifold $\tilde{X}=\tilde{G}/\tilde{\Gamma}$, where
 $\tilde{G}$ is a connected and simply-connected nilpotent Lie group, with all translations from
 $G$   ``represented'' in $\tilde{G}$  (for example if $X=\T$ then $\tilde{X}=\R/\Z$, and if
   $X=(\Z\times\R^2)/\Z^3$ then
 $\tilde{X}=\R^3/\Z^3$). Practically,  this means that for every  $F\in C(X)$, $b\in G$, and $x\in X$,   there exists
 $\tilde{F}\in C(\tilde{X})$, $\tilde{b}\in \tilde{G}$, and $\tilde{x}\in \tilde{X}$, such that
 $F(b^nx)=\tilde{F}(\tilde{b}^n\tilde{x})$ for every $n\in\N$.

One should keep in mind though  when using this lifting trick, that any assumption made about a nilrotation
 $b$ acting on a nilmanifold $X$, is typically lost when passing to the lifted nilmanifold $\tilde{X}$.
 Therefore,  the above mentioned construction will be helpful only when our working assumptions impose no
 restrictions on a nilrotation.

 \begin{example}\label{Ex2}
  Let $G$ be the non-connected nilpotent group that consists of all upper triangular matrices of the form
  $\left(\begin{smallmatrix} 1& k& z\\ 0& 1& y\\0& 0 & 1\end{smallmatrix}\right)$ where $k\in \Z$ and $y,z\in \R$.
  If we also restrict the entries $y$ and $z$ to be  integers  we get a subgroup  $\Gamma$
  of $G$ that is discrete and cocompact. In this case, the Heisenberg nilmanifold of Example~\ref{Ex1}
  can serve as the lifting  $\tilde{X}$  of the nilmanifold $X=G/\Gamma$.
  \end{example}

  \subsection{Equidistribution on nilmanifolds}\label{SS:nilequi}
  We
gather some equidistribution results of polynomial sequences on
nilmanifolds that will be used later.

\subsubsection{Qualitative equidistribution on nilmanifolds}
If $G$ is a nilpotent group, then a sequence $g\colon \N\to G$ of the
form $g(n)=a_1^{p_1(n)}\ \! a_2^{p_2(n)}\cdots  a_k^{p_k(n)}$, where $a_i\in G$,
and $p_i$ are polynomials taking integer values at the integers, is
called a \emph{polynomial sequence in} $G$. If the maximum
degree of the polynomials $p_i$ is at most $d$ we say that the
\emph{degree} of $g(n)$ is at most $d$.
A \emph{polynomial sequence on the
  nilmanifold} $X=G/\Gamma$ is a sequence of the form
$(g(n)\Gamma)_{n\in\N}$ where $g\colon \N\to G$ is a polynomial
sequence in $G$.

\begin{theorem}[{\bf Leibman~\cite{L2}}]\label{T:L}
  Suppose that $X = G/\gG$ is a nilmanifold, with  $G$ connected and simply connected,
   and $(g(n))_{n\in \N}$ is a
  polynomial sequence in $G$. Let $Z=G/([G,G]\Gamma)$ and
  $\pi\colon X\to Z$ be the natural projection.

Then the following statements are true:

  $(i)$ The
  sequence $(g(n)x)_{n\in\mathbb{N}}$ is equidistributed in a finite union of
  sub-nilmanifolds of $X$.

  $(ii)$ For every  $x\in X$ the
  sequence $(g(n)x)_{n\in\mathbb{N}}$ is equidistributed in $X$ if and
  only if the sequence  $(g(n)\pi(x))_{n\in\mathbb{N}}$ is equidistributed in $Z$.
\end{theorem}




\subsubsection{Quantitative equidistribution on nilmanifolds}
We shall frequently use a quantitative version of Theorem~\ref{T:L}
that was  obtained
in \cite{GT}. In order to state it we need to review
some notions that were introduced in \cite{GT}.

Given a nilmanifold $X=G/\Gamma$, the {\em horizontal torus} is
defined to be the compact Abelian group $Z=G/([G,G]\Gamma)$.  If $X$
is connected, then $Z$ is isomorphic to some finite dimensional torus
$\T^l$. By $\pi\colon X\to H$ we denote the natural projection map.  A
\emph{horizontal character} $\chi\colon G\to \C$ is a continuous homomorphism
 that
satisfies $\chi(g\gamma)=\chi(g)$ for every $\gamma\in\Gamma$. Since
every character annihilates $G_2$, every horizontal character factors
through $Z$, and therefore can be thought of as a character of the horizontal
torus. Since $Z$ is identifiable with a finite dimensional torus
$\T^l$ (we assume that $X$ is connected), $\chi$ can also be thought
of as a character of $\T^l$, in which case there exists a unique
$\kappa\in\Z^l$ such that $\chi(t\Z^l)=e(\kappa\cdot t)$, where $\cdot$
denotes the inner product operation.
 We refer to $\kappa$ as the
\emph{frequency} of $\chi$ and $\norm{\chi}=|\kappa|$ as the \emph{frequency
  magnitude} of $\chi$.

\begin{example}
  Let $X$ be the Heisenberg nilmanifold (see Example~\ref{Ex1}).  The map $\chi \colon G\to \C$ defined by
  $\chi \left(\begin{smallmatrix} 1& x& z\\ 0& 1& y\\0& 0 & 1\end{smallmatrix}\right)=e(kx+ly)$,
  where $k,l\in\Z$, is a horizontal character of $G$. The map
  $\phi\left(\begin{smallmatrix} 1& x& z\\ 0& 1& y\\0& 0 & 1\end{smallmatrix}\right)
  =(x\Z,y\Z)$ induces an identification of the
  horizontal torus with $\T^2$.  Under this identification, $\chi$ is
  mapped to the character $\tilde{\chi}(x\Z, y\Z)=e(kx+ly)$ of $\T^2$.
\end{example}

Suppose that $p\colon\Z\to \R$ is a polynomial sequence of degree $k$,
then $p$ can be uniquely expressed in the form $
p(n)=\sum_{i=0}^k\binom{n}{i}\alpha_i $ where $\alpha_i\in\R$. We
define
\begin{equation}\label{E:norms}
  \norm{e(p(n))}_{C^\infty[N]}=\max_{1\leq i\leq k}( N^i \norm{\alpha_i})
\end{equation}
where $\norm{x}=d(x,\Z)$.


Given $N\in\N$, a finite sequence $(g(n)\Gamma)_{1\leq n\leq N}$ is
said to be $\delta$-\emph{equidistributed}, if
$$
\Big|\frac{1}{N}\sum_{n=1}^N F(g(n)\Gamma)-\int_{X}F \ dm_X\Big|\leq
\delta \norm{F}_{\text{Lip}(X)}
$$
for every Lipschitz function $F\colon X\to \C$, where
$$
\norm{F}_{\text{Lip}(X)}=\norm{F}_\infty+ \sup_{x,y\in X, x\neq
  y}\frac{|F(x)-F(y)|}{d_X(x,y)}
$$
for some appropriate metric $d_X$ on $X$.
  We can now
state the equidistribution result that  we shall use. It is a direct consequence
of Theorem~2.9 in \cite{GT} (we have suppressed some distracting quantitative details
that  will be of no use for us):
\begin{theorem}[{\bf Green \& Tao~\cite{GT}}]\label{T:GT}
  Let $X=G/\Gamma$ be a nilmanifold with $G$ connected and simply
  connected, and  $d\in\N$.

 Then for every small  enough $\delta>0$ there exist  $M=M_{X,d,\delta}\in \R$
  with the following property:
  For every
  $N\in\N$,
   if $g\colon \Z\to G$ is a polynomial sequence of degree at most $d$
  such that the finite sequence $(g(n)\Gamma)_{1\leq n\leq N}$ is not
  $\delta$-equidistributed,
    then for some non-trivial horizontal character
  $\chi$ with  $\norm{\chi}\leq M$  we have
  \begin{equation}\label{E:badequi}
    \norm{\chi( g(n))}_{C^\infty[N]}\leq  M,
  \end{equation}
   where $\chi$ is thought of as a
  character of the horizontal torus $Z=\T^l$ and $g(n)$ as a
  polynomial sequence in $\T^l$.
\end{theorem}


\begin{example}\label{Ex:A}
  It is instructive to interpret the previous result in some special
  case. Let $X=\T$ (with the standard metric), and suppose that the polynomial
  sequence on $\T$ is given by $p(n)=(n^d\alpha+q(n))\Z$ where $d\in \N$,
  $\alpha\in \R$, and $q\in\Z[x]$ with $\deg(q)\leq d-1$. In this case
  Theorem~\ref{T:GT} reads as follows: There exists $M>0$ such
  that for every $N\in\N$ and $\delta$ small enough,
  if the finite sequence $\big((n^d\alpha+q(n))\Z\big)_{1\leq n\leq N}$ is not
  $\delta$-equidistributed in $\T$, then $\norm{k\alpha}\leq
  M/N^d$ for some non-zero $k\in \Z$ with $|k|\leq M$.
\end{example}


\section{A model equidistribution result}\label{S:warmup}
Before delving into the proof of   the various equidistribution results on nilmanifolds  we find it instructive to
deal with a much simpler  equidistribution problem on the circle.
 This model problem  will motivate some of the ideas used later.
We shall give a new proof for the following  result:
\begin{theorem}[{\bf Boshernitzan} \cite{Bo2}]\label{T:Bos}
Let $a\in \H$ have  polynomial growth.

 Then the sequence $(a(n)\Z)_{n\in\N}$ is equidistributed in $\T$
if and only if  for every $p\in \Q[t]$  we have $|a(t)-p(t)|\succ \log t$.
\end{theorem}
Our strategy  will be to use the Taylor expansion of the function $a(t)$ to partition the range of the sequence $(a(n))_{n\in\N}$ into blocks that
are approximately polynomial and then use classical results to  estimate the corresponding exponential sums over
these blocks.
This argument can be adapted to the non-Abelian setup we are interested in,
 the main reason being  that  ``Weyl type" sums involving polynomial block  sequences of fixed degree
on nilmanifolds can be effectively estimated
using a rather sophisticated application of the van der Corput difference trick (this is done in \cite{GT}), and with a
 bit of care  one can  piece together these estimates to get usable results. The following simple example best illustrates our method:
\begin{example}\label{Ex:nlogn}
Suppose that $a(t)=t \log{t}$. We shall show that the sequence  $(n\log{n}\, \Z)_{n\in \N}$ is equidistributed in $\T$.
Using Lemma~\ref{L:averaging} below, it suffices to show that for every non-zero integer $k$ we have
$$
\lim_{N\to \infty}\E_{N< n\leq  N+N^{\frac{3}{5}}}e(kn\log{n})=0.
$$
For convenience we assume that  $k=1$.

 Using the Taylor expansion of $a(t)$ around the point $x=N$ we see that for $n\in [1, N^{3/5}]$ we have
\begin{equation}\label{E:nlogn}
(N+n)\log(N+n)=N \log{N}+(\log{N}+1)n+\frac{n^2}{2N}+o_{N\to\infty}(1).
\end{equation}
(Notice that we keep track of the smallest order derivative
that converges to zero and drop higher order derivatives.)
Using \eqref{E:nlogn} one gets
\begin{equation}\label{E:hgv}
\E_{N< n\leq  N+N^{\frac{3}{5}}}e(n\log{n})= \E_{1\leq n\leq  N^{\frac{3}{5}}}e\Big(N \log{N}+(\log{N}+1)n+\frac{n^2}{2N}\Big)+o_{N\to\infty}(1).
\end{equation}
Since
$
(N^{\frac{3}{5}})^2\norm{\frac{1}{2N}}\to \infty
$,  using  Weyl's estimates  (see e.g. \cite{Va97}) we get
 that the averages in \eqref{E:hgv} converge to $0$ as $N\to \infty$.
\end{example}
\begin{lemma}\label{L:averaging}
Let $(a(n))_{n\in\N}$ be a bounded sequence of complex numbers.  Suppose that
$$
\lim_{N\to\infty}\big(\E_{N\leq n\leq N+l(N)}a(n)\big)=0
$$
for some positive function $l(t)$ with $l(t)\prec t$.
Then
$$
\lim_{N\to\infty} \E_{1\leq n\leq N} a(n)=0.
$$
\end{lemma}
\begin{proof}
We can cover the interval $[1,N]$ by a union $I_N$  of non-overlapping intervals of the
 form $[k,k+l(k)]$.
 Since $l(t)\prec t$ and the sequence
  $(a(n))_{n\in\N}$ is bounded, we  have that
$$
\lim_{N\to \infty} \E_{1\leq n\leq N}a(n)=\lim_{N\to \infty} \E_{n\in I_N}a(n).
$$
 Using our assumption, one easily gets that
 the limit $\lim_{N\to \infty} \E_{n\in I_N}a(n)$ is zero, finishing the proof.
\end{proof}
A  modification of the argument used in Example~\ref{Ex:nlogn} gives the following more general result:
\begin{lemma}\label{L:basicequi}
Suppose that for some $m\in \N$ the function $a\in C^{m+1}(\R_+)$ satisfies
$$
|a^{(m+1)}(t)| \text{ is decreasing}, \quad 1/t^m\prec a^{(m)}(t)\prec 1,
\quad    (a^{(m+1)}(t))^m\prec (a^{(m)}(t))^{m+1}.
$$

Then the sequence
  $(a(n)\Z)_{n\in\N}$ is equidistributed in $\T$.
\end{lemma}
\begin{proof}
It suffices to show that for every non-zero integer $k$ we have
$$
\lim_{N\to\infty}\E_{1\leq n\leq N} e(ka(n))=0.
$$
Since our assumptions are also satisfied for $ka(t)$ in place of $a(t)$ whenever $k\neq 0$,
we can assume that $k=1$.

By Lemma~\ref{L:averaging} it is enough to show that the averages
\begin{equation}\label{E:bhn1}
\E_{N\leq n\leq N+l(N)} e(a(n))
\end{equation}
converge to zero as $N\to \infty$ for some positive function $l(t)$ that satisfies $l(t)\prec t$.\footnote{The choice
 of $l(t)$ will depend on the function $a(t)$.
For example, if $a(t)=t\log{t}$ we need to assume that  $t^{1/2}\prec l(t)\prec t^{2/3}$, and if  $a(t)=(\log{t})^2$ we need
to assume that $t/\log{t}\prec l(t)\prec t/\sqrt{\log{t}}$.}

 Using the Taylor expansion of $a(t)$ around the point $t=N$ we get
\begin{equation}\label{E:Taylor1}
a(N+n)=a(N)+na'(N)+\cdots +\frac{n^m}{m!}a^{(m)}(N)+\frac{n^{m+1}}{(m+1)!}a^{(m+1)}(\xi_n)
\end{equation}
for some $\xi_n\in [N,N+n]$. Since $|a^{(m+1)}(t)|$ is  decreasing we have $|a^{(m+1)}(\xi_n)|\leq |a^{(m+1)}(N)|$.
It follows that if $l(t)$  also satisfies
$$
(l(t))^{m+1} a^{(m+1)}(t)\prec 1,
$$
  then  the averages in \eqref{E:bhn1}
 are equal to
$$
 \E_{1\leq n\leq l(N)} e\Big(a(N)+na'(N)+\cdots +\frac{n^m}{m!}a^{(m)}(N)\Big) +o_{N\to\infty}(1).
$$
Next, using   Example~\ref{Ex:A}  (or Weyl's estimates; see e.g. \cite{Va97}) we get
 that the last averages converge to zero as $N\to \infty$ if
$$
1\prec
(l(t))^m \norm{a^{(m)}(t)}=
(l(t))^m |a^{(m)}(t)|,
$$
the last equality being valid for every large $t$  since $|a^{(m)}(t)|\to 0$.

Summarizing, we have shown that the averages in \eqref{E:bhn1} converge to zero when $N\to \infty$
as long as we can establish the existence of
a function
$l(t)$ satisfying the following conditions
\begin{equation}\label{E:pro}
l(t)\prec t\  \text{ and }   \    (l(t))^{m+1} |a^{(m+1)}(t)|\prec 1\prec (l(t))^m |a^{(m)}(t)|.
\end{equation}
Since by assumption $(a^{(m+1)}(t))^{\frac{m}{m+1}}\prec  a^{(m)}(t)$ and $1/t^m\prec a^{(m)}(t)$,
we can indeed find a function $l(t)$ that satisfies $\max\big((a^{(m+1)}(t)\big)^{\frac{m}{m+1}},1/t^m)\prec 1/(l(t))^m\prec |a^{(m)}(t)|$, and so \eqref{E:pro} holds.
This completes the proof.
\end{proof}
The previous lemma  applies to a wide variety of functions.
For example the functions  $(\log{t})^2$,  $t\log{t}$, $t^{3/2}$,
 $t^2\sqrt{2}+t^{1/2}$  satisfy the stated assumptions.
In fact our next lemma shows that Lemma~\ref{L:basicequi} comes rather close to establishing
Theorem~\ref{T:Bos}.
\begin{lemma}\label{L:cvb}
 Let $a\in\H$ have polynomial growth and satisfy $|a(t)-p(t)|\succ \log{t}$  for every $p\in\R[t]$.

  Then the function $a(t)$ satisfies the assumptions of Lemma~\ref{L:basicequi} for some $m\in\N$. As a consequence, the sequence
  $(a(n)\Z)_{n\in\N}$ is equidistributed in $\T$.
\end{lemma}
\begin{proof}
By Lemma~\ref{L:type} the function  $a(t)$ has type  $k$ or $k^+$ for some non-negative integer $k$.
We shall show that  the assumptions of Lemma~\ref{L:basicequi} are satisfied for
$m= k+1$. We can assume that the function  $a^{(k)}(t)$ is eventually positive,  if this  not the case
we work with the function  $-a(t)$.

Since $a(t)\prec t^{k+1}$, it follows  from  Corollary~\ref{C:properties} that the functions $a^{(k+1)}(t)$ and
 $a^{(k+2)}(t)$ converge to zero. Furthermore, since both functions are elements of $\H$  the convergence is monotone.

We show that $a^{(k+1)}(t)\succ 1/t^{k+1}$. Suppose first  that $a(t)$ has type $k^+$ for some positive integer $k$.
By Corollary~\ref{C:properties} we have
$$
a^{(k+1)}(t)\succ a(t)/(t^{k+1}(\log{t})^2)\gg 1/\big(t(\log{t})^2\big)\succ 1/t^{k+1}.
$$
Suppose now that  $a(t)$ has type  $0^+$, in which case we shall show that $a'(t)\succ 1/t$.
Arguing by contradiction, suppose that
this is not the case. Since $a'(t)$ is eventually positive,  we conclude that for large values of $t$ we have
$0\leq a'(t)\leq  c_1/t$ for some non-negative constant $c_1$.
 Integrating we get that for large values of $t$ we have  $0\leq a(t) \leq c_1\log{t}+c_2$ for some constants $c_1,c_2$, contradicting our assumption
 $|a(t)|\succ \log{t}$.
Lastly, suppose that $a(t)$ has type $k$ for some non-negative integer $k$. Since $a(t)$ stays away from polynomials,
we  conclude from Lemma~\ref{L:type} that    $a(t)=p(t)+b(t)$,
 for some $p\in \R[t]$ of degree $k$, and some $b\in \H$
 of type $l^+$ for some non-negative integer $l$ with $l<k$. Arguing
 as  before, we conclude that
  $b^{(k+1)}(t)\succ 1/t^{k+1}$. Since
 $a^{(k+1)}(t)=b^{(k+1)}(t)$, we get $a^{(k+1)}(t)\succ 1/t^{k+1}$.

It remains to show that
$(a^{(k+2)}(t))^{k+1}\prec(a^{(k+1)}(t))^{k+2}$.
 By  Lemma~\ref{L:properties}
we know that $a^{(k+2)}(t)\ll a^{(k+1)}(t)/t$.  Using this, and the previously established estimate
$a^{(k+1)}(t)\succ 1/t^{k+1},$ we get
$$
(a^{(k+2)}(t))^{k+1} \ll (a^{(k+1)}(t))^{k+1}/t^{k+1}\prec (a^{(k+1)}(t))^{k+2}.
$$
This completes the proof.
\end{proof}
We now complete the proof of Theorem~\ref{T:Bos}
\begin{proof}[Proof of Theorem~\ref{T:Bos}]
We first prove the sufficiency of the conditions.  Combining Lemma~\ref{L:basicequi} and   Lemma~\ref{L:cvb}
we cover the case where  $|a(t)-p(t)|\succ \log{t}$ for every $p\in \R[t]$.
It remains  to deal with the case where $a(t)=p(t)+e(t)$ for some $p\in\R[t]$ that has at
 least one non-constant coefficient irrational and $e(t)\ll \log{t}$. Since $e(n+1)-e(n)\to 0$ (this follows from the mean value theorem and the fact that $e'(t)\to 0$), we can write  $\N$ as a union of non-overlapping intervals $(I_m)_{m\in\N}$ such that  $|I_m|\to \infty$ and $\max_{n_1,n_2\in I_m}|e(n_1)-e(n_2)|\leq 1/m$. Combining this with the fact that the sequence $(p(n)\Z)_{n\in\N}$ is well distributed in $\T$
 (meaning $\lim_{N-M\to\infty} \E_{M\leq n\leq N} e(kp(n))=0$ for every non-zero $k\in\Z$), we deduce that the sequence $(a(n))_{n\in\N}$ is equidistributed in $\T$.

To prove the necessity of the conditions suppose that $ 1\prec a(t)\ll \log{t}$; the general case can be easily reduced to this one.
We shall show that the sequence $(a(n))_{n\in\N}$ cannot be equidistributed in $\T$.
The key property we shall use
is that  $a(n+1)-a(n)\ll 1/n$. (This
estimate is a consequence of the mean value theorem and the estimate
 $a'(t)\leq  c/t$ for large enough $t$ which can be proved as in  Lemma~\ref{L:properties}.)
 For convenience we assume that
 $a(n+1)-a(n)< 1/n$ is satisfied for every $n\in \N$, and the sequence $(a(n)\Z)_{n\in\N}$ is increasing. The general case is similar.
Arguing by contradiction, suppose that the sequence $(a(n)\Z)_{n\in\N}$ is equidistributed. Let
 $n_m$ be the first integer that satisfies $a(n_m)>m$. Since  $a(n_m)<a(n)<a(n_m)+ n/n_m$ and $a(n_m)$ is
 very close to an integer for large $m$,
  approximately all the integers in $[n_m,3n_m/2]$ satisfy $\{a(n)\}\leq 1/2$.
  Furthermore, because of the equidistribution
property, for large $m\in \N$,  approximately  half
of the integers in $[1,n_m]$
 satisfy $\{a(n)\}\leq 1/2$.  Therefore, for large $m\in\N$,  approximately two thirds of the integers
  in $[1,3n_m/2]$ satisfy $\{a(n)\}\leq 1/2$, contradicting our equidistribution assumption.
\end{proof}

\section{Single nil-orbits and Hardy  sequences}\label{S:single}
In this section we are going to prove Theorems~\ref{T:A} and \ref{T:B}.
\subsection{A reduction}\label{SS:reduction}
We start with some
initial maneuvers that will allow  us to reduce Theorem~\ref{T:B} to a more convenient statement.

First we give a
result that enables us to translate
  distributional  properties of sequences of the form  $(b^{a(n)}x)_{n\in\N}$ to sequences of the form $(b^{[a(n)]}x)_{n\in\N}$.
\begin{lemma}\label{L:reduction}
Let $(a(n))_{n\in\N}$ be a sequence of real numbers such that for every nilmanifold $X=G/\Gamma$, with $G$ connected and simply connected,
and  every $b\in G$,
the sequence $(b^{a(n)}\Gamma)_{n\in\N}$ is equidistributed in the nilmanifold $ \overline{(b^s\Gamma)}_{s\in \R}$.

Then for every nilmanifold $X=G/\Gamma$,  every   $b\in G$ and $x\in X$,
the sequence $(b^{[a(n)]}x)_{n\in\N}$ is equidistributed in the nilmanifold $ \overline{(b^nx)}_{n\in \N}$.
\end{lemma}
\begin{proof}

Let $X=G/\Gamma$ be a nilmanifold $b\in G$ and $x\in X$.
We start with some reductions. By using the lifting argument of Section~\ref{SS:nil}, we can assume that $G$
is connected and simply connected. Furthermore, by changing the base point and using the formula in Section~\ref{SS:nil}, we can assume that $x=\Gamma$.

 Let $X_b$ be the nilmanifold $\overline{(b^n\Gamma)}_{n\in\N}$
and $m_{X_b}$ be the corresponding normalized Haar measure. It suffices to  show that for every $F\in C(X)$ we have
\begin{equation}\label{E:[a(n)]}
\lim_{N\to\infty}\E_{1\leq n\leq N}F(b^{[a(n)]}\Gamma)=\int_{X_b} F\ dm_{X_b}.
\end{equation}

 So let $F\in C(X)$.
To begin with, we use our assumption in the following case
$$
\tilde{X}=\tilde{G}/\tilde{\Gamma} \ \text{ where }\
 \tilde{G}=\R\times G, \ \tilde{\Gamma}=\Z\times\Gamma, \ \text{
and }  \ \tilde{b}=(1,b).
$$ (Notice that $\tilde{G}$ is  connected and simply connected.)
We conclude that  for every $\tilde{H}\in C(\tilde{X})$
\begin{equation}\label{E:continuous}
\lim_{N\to\infty}\E_{1\leq n\leq N} \tilde{H}(\tilde{b}^{a(n)}\tilde{\Gamma})=\int_{\tilde{X}_{\tilde{b}}} \tilde{H} \ dm_{\tilde{X}_{\tilde{b}}},
\end{equation}
where $\tilde{X}_{\tilde{b}}$ is the nilmanifold $\overline{(s\Z,b^s\Gamma)}_{s\in\R}$, and  $m_{\tilde{X}_{\tilde{b}}}$ is the corresponding normalized Haar measure.

Next we claim that \eqref{E:continuous} can be applied for the  function $\tilde{F}\colon \tilde{X}\to \C$ defined by
 \begin{equation}\label{E:tilde{F}}
  \tilde{F}(t\Z,g\Gamma)=F(b^{-\{t\}}g\Gamma).
  \end{equation}
 We caution the reader that the function $\tilde{F}$ may be    discontinuous.
The set of discontinuities of $\tilde{F}$  is a subset of the sub-nilmanifold $\{\Z \}\times X$. Near
a point $(\Z,g\Gamma)$ of  $\{\Z \}\times X$ the function  $\tilde{F}$ comes close to the value
$F(g\Gamma)$ or the value $F(b^{-1}g\Gamma)$.
 For  $\delta>0$  (and smaller than $1/2$) there exist functions  $\tilde{F}_\delta\in C(\tilde{X})$ that agree with $\tilde{F}$ on $\tilde{X}_\delta=I_\delta\times X$, where $I_\delta=\{t\Z\colon \norm{t}\geq \delta\}$, and are uniformly bounded by $2\norm{F}_\infty$.
   Our assumption gives  that  the sequence $(a(n)\Z)_{n\in\N}$ is equidistributed in $\T$. Since $\tilde{b}^{a(n)}=(a(n),b^{a(n)})$,  we deduce that      $\tilde{b}^{a(n)}\tilde{\Gamma}\in \tilde{X}_\delta$ for a set of  $n\in\N$ with density $1-2\delta$. As a consequence,
\begin{equation}\label{E:approx0}
\limsup_{N\to\infty} \E_{1\leq n\leq N}|\tilde{F}(\tilde{b}^{a(n)}\tilde{\Gamma})- \tilde{F}_\delta(\tilde{b}^{a(n)}\tilde{\Gamma})|\leq 4\norm{F}_\infty\delta.
\end{equation}
By assumption, \eqref{E:continuous} holds when one uses the functions $\tilde{F}_\delta$ in place of the function $\tilde{H}$.  Using these identities for every $\delta>0$, and letting $\delta\to 0$, we get using    \eqref{E:approx0} that \eqref{E:continuous} also holds for the discontinuous function $\tilde{F}$  defined in \eqref{E:tilde{F}} (to get that
$\int \tilde{F}_\delta \ dm_{\tilde{X}_{\tilde{b}}}\to \int \tilde{F} \ dm_{\tilde{X}_{\tilde{b}}}$ we use that
$m_{\tilde{X}_{\tilde{b}}}(\{0\}\times X)=0$, which holds since $\{0\}\times X$ is a proper sub-nilmanifold of $\tilde{X}_{\tilde{b}}$). This  verifies our claim.


Applying \eqref{E:continuous} for the function $\tilde{F}$ defined in \eqref{E:tilde{F}}, and noticing that
$$
\tilde{F}(\tilde{b}^{a(n)}\tilde{\Gamma})=F(b^{-\{a(n)\}}b^{a(n)}\Gamma)=F(b^{[a(n)]}\Gamma),
$$
we get
$$
\lim_{n\to\infty}\E_{1\leq n\leq N}F(b^{[a(n)]}\Gamma)=\int_{\tilde{X}_{\tilde{b}}} \tilde{F} \ dm_{\tilde{X}_{\tilde{b}}}=\int_{\tilde{X}_{\tilde{b}}} F(b^{-\{s\}}g\Gamma) \ dm_{\tilde{X}_{\tilde{b}}}(s\Z,g\Gamma).
$$
Since $b^{-\{s\}}b^s\Gamma=b^{[s]}\Gamma$,  the map $(s\Z,g\Gamma)\rightarrow b^{-\{s\}}g\Gamma$ sends  the nilmanifold $\tilde{X}_{\tilde{b}}$ onto
the nilmanifold $X_b=\overline{(b^n\Gamma)}_{n\in\N}$. On   $X_b$  we define the measure $m$ by letting
$$
\int_{X_b}F\ dm=\int_{\tilde{X}_b} F(b^{-\{s\}}g\Gamma) \ dm_{\tilde{X}_{\tilde{b}}}(s\Z,g\Gamma)
$$
for every $F\in C(X_b)$. We claim that $m=m_{X_b}$. Indeed,
a quick computation shows that the measure $m$ is invariant under left translation by $b$.
As it is well known, any rotation
 $b$ is uniquely ergodic on its orbit closure $X_b$,  hence $m=m_{X_b}$. This establishes \eqref{E:[a(n)]}
 and completes the proof.
\end{proof}
The previous lemma shows that part $(ii)$ of Theorem~\ref{T:B} follows from part $(i)$.
 It turns out that dealing with part $(i)$ presents significant technical advantages (in fact we do not see how to
 establish part $(ii)$ directly).


Next we show that in order to prove part $(i)$ of Theorem~\ref{T:B} it suffices to establish the following result:
\begin{proposition}\label{P:B}
Let $a\in\H$ have polynomial growth and satisfy $|a(t)-cp(t)|\succ \log{t}$ for every $c\in \R$ and $p\in \Z[t]$.
Let $X=G/\Gamma$ be a nilmanifold, with $G$
connected and simply connected, and suppose that  $b\in G$  acts ergodically on $X$.

Then
 the sequence $(b^{a(n)}\Gamma)_{n\in\N}$ is  equidistributed in $X$.
\end{proposition}
To carry out this reduction we shall use the following lemma:
\begin{lemma}\label{L:ergodic}
 Let  $X=G/\Gamma$ be a nilmanifold with $G$
connected and simply connected.

Then
for every $b\in G$ there exists $s_0\in \R$ such that  the element $b^{s_0}$ acts ergodically
 on the nilmanifold $ \overline{(b^s\Gamma)}_{s\in \R}$.
\end{lemma}
\begin{proof}
By Ratner's theorem (see Section~\ref{SS:nil}), we have   $\overline{(b^s\Gamma)}_{s\in \R}=H/\Delta$, where  $H$ is a connected
and simply connected closed subgroup of $G$ that contains all the elements $b^s$, $s\in \R$, and $\Delta=H\cap\Gamma$.
By Theorem~\ref{T:L}  it suffices to check that  $b^{s_0}$ acts
ergodically on the horizontal torus $G/([G,G]|\Gamma)$, which we can assume to be $\T^k$ for some $k\in\N$.
 Equivalently, this amounts to showing that if $\beta\Z^k\in \T^k$, where $\beta\in \R^k$,
 then there exists $s_0\in\R$ such that
$\overline{(ns_0 \beta\Z^k)}_{n\in\N}=\overline{(s\beta\Z^k)}_{s\in\R}$. One can  check (we omit the routine details) that it  suffices to choose $s_0$ such that the number $1/s_0$ is rationally independent of any non-zero integer combination of the coordinates of $\beta$. This completes the proof.
\end{proof}
Putting together Lemma~\ref{L:reduction} and Lemma~\ref{L:ergodic} we get the advertised reduction:
\begin{proposition}\label{P:reduction}
In order to prove Theorem~\ref{T:B} it suffices to prove Proposition~\ref{P:B}.
\end{proposition}
\begin{proof}
Using  Lemma~\ref{L:reduction}, we see that part $(ii)$ of Theorem~\ref{T:B} follows from part $(i)$.

To establish part $(i)$ we argue as follows. Let $b\in G$. By Lemma~\ref{L:ergodic} there exists non-zero
 $s_0\in\R$ such that the element $b^{s_0}$ acts ergodically
 on the nilmanifold $ \overline{(b^{s}\Gamma)}_{s\in \R}$. Using  Proposition~\ref{P:B} for the element $b^{s_0}$
 and the function $a(s)/{s_0}$, we get that the sequence $(b^{a(n)}\Gamma)_{n\in\N}$ is equidistributed in the nilmanifold
 $ \overline{(b^s\Gamma)}_{s\in \R}$.
\end{proof}

 We now turn our  attention   to the proof of Proposition~\ref{P:B}.


\subsection{Proof of Proposition~\ref{P:B}}
The following lemma is the key ingredient in the proof of  Proposition~\ref{P:B}:
\begin{lemma}\label{L:uio} Suppose that for some $k\in \N$ the function $a\in C^{k+1}(\R_+)$ satisfies
$$
|a^{(k+1)}(t)| \text{ is decreasing}, \quad 1/t^k\prec a^{(k)}(t)\prec 1,
\quad    (a^{(k+1)}(t))^k\prec (a^{(k)}(t))^{k+1}.
$$
  Let $X=G/\Gamma$ be a nilmanifold, with $G$
connected and simply connected, and suppose that  $b\in G$ acts ergodically on $X$.

Then
the sequence $(b^{a(n)}\Gamma)_{n\in\N}$ is  equidistributed in $X$.
\end{lemma}
\begin{proof}
Let $F\in C(X)$ with zero integral. We want  to show that
$$
\lim_{N\to\infty} \E_{1\leq n\leq N} F(b^{a(n)}\Gamma)=0.
$$
By Lemma~\ref{L:averaging}  it suffices to show that the averages
\begin{equation}\label{E:bhn1'}
 \E_{N\leq n\leq N+l(N)} F(b^{a(n)}\Gamma)
\end{equation}
converge to zero as $N\to \infty$ for some positive function $l(t)$ that satisfies $l(t)\prec t$.

Using the Taylor expansion of $a(t)$ around the point $x=N$ we have
\begin{equation}\label{E:Taylor1'}
a(N+n)=a(N)+na'(N)+\cdots +\frac{n^k}{k!}a^{(k)}(N)+\frac{n^{k+1}}{(k+1)!}a^{(k+1)}(\xi_n)
\end{equation}
for some $\xi_n\in [N,N+n]$. Since $|a^{(k+1)}(t)|$ is  decreasing we have $|a^{(k+1)}(\xi_n)|\leq |a^{(k+1)}(N)|$.
It follows that if the function $l(t)$ satisfies
$$
(l(t))^{k+1} a^{(k+1)}(t)\prec 1,
$$
 then the  averages
 \eqref{E:bhn1'} are equal to
$$
 \E_{1\leq n\leq l(N)} F\Big(b^{p_N(n)}\Gamma\Big) +o_{N\to\infty}(1)
$$
where
$$
p_N(n)=a(N)+na'(N)+\cdots +\frac{n^k}{k!}a^{(k)}(N).
$$
Our objective now is to show that for every $\delta>0$, for large values of  $N$, the finite sequence
$(b^{p_{N}(n)}\Gamma)_{1\leq n\leq l(N)}$  is $\delta$-equidistributed in $X$. This would immediately imply that
the averages in \eqref{E:bhn1'} converge to zero as $N\to \infty$.

So let $\delta>0$. Notice first that since
$$
b^{p_{N}(n)}=b_{0,N}\  \! b_{1,N}^n\cdots b_{k,N}^{n^k},
$$
where $b_{i,N}=b^{a^{(i)}(N)/i!}$ for $i=0,1,\ldots,k$, for every  fixed $N\in\N$ the sequence $(b^{p_{N}(n)})_{n\in\N}$
is a polynomial sequence in $G$.
 Since $X=G/\Gamma$ and  $G$ is connected and
simply connected, we can apply  Theorem~\ref{T:GT} (for $\delta$ small enough). We conclude that
if the finite sequence $(b^{p_{N}(n)}\Gamma)_{1\leq n\leq l(N)}$ is not $\delta$-equidistributed in $X$, then there exists
 a constant $M$ (depending only on $\delta$, $X$, and  $k$), and a horizontal character $\chi$ with $\norm{\chi}\leq M$ such that
\begin{equation}\label{E:C_R}
\norm{\chi(b^{p_N(n)})}_{C^\infty[l(N)]}\leq M.
\end{equation}
Let $\pi(b)=(\beta_1\Z,\ldots,\beta_s\Z)$, where $\beta_i\in \R$,  be the projection of $b$ on the horizontal torus $\T^s$ (notice that
$s$ is bounded by the dimension of $X$).
Since $b$ acts ergodically on $X$ the real numbers $1, \beta_1,\ldots,\beta_s$ must be rationally independent.
For $t\in \R$ we have $\pi(b^t)=(t\tilde{\beta}_1\Z,\ldots, t\tilde{\beta}_s\Z)$ for some
 $\tilde{\beta}_i\in \R$ with  $\tilde{\beta}_i\Z=\beta_i\Z$.
As a consequence, we have
$$
\chi(b^{p_N(n)})=e\Big(p_N(n) \sum_{i=1}^s l_i\tilde{\beta}_i\Big)
$$
for some $l_i\in\Z$ with $|l_i|\leq M$.
From this, the definition of $p_N(t)$, and the definition of $\norm{\cdot}_{C^\infty[N]}$ (see \eqref{E:norms}),  we get that
$$
\norm{\chi(b^{p_N(n)})}_{C^\infty[l(N)]}\geq (l(N))^k\norm{a^{(k)}(N)\beta},
$$
where  $\beta$ is a non-zero (we use the rational independence of the $\tilde{\beta}_i$'s here) real number that belongs to the finite set
$$
B=\Big\{\frac{1}{k!}\sum_{i=1}^s l_i\tilde{\beta}_i\colon |l_i|\leq M\Big\}.
$$
 Combining this estimate with \eqref{E:C_R}, and using that $\norm{a^{(k)}(N)\beta}=|a^{(k)}(N)\beta|$ for
 large $N$ (since by assumption $a^{(k)}(t)\to 0$), we get
\begin{equation}\label{E:bro}
(l(N))^k|a^{(k)}(N)\beta|\leq M
\end{equation}
for some $\beta\in B$. It follows that if the function $l(t)$ satisfies
$$
1\prec (l(t))^k a^{(k)}(t),
$$
then \eqref{E:bro} fails for large $N$, and as a result
 the finite sequence $(b^{p_{N}(n)}\Gamma)_{1\leq n\leq l(N)}$ is $\delta$-equidistributed
 in $X$ for every large $N$.

Summarizing, we have shown that  the averages \eqref{E:bhn1'} converge to $0$ when  $N\to \infty$,
 as long as we can find a positive function $l(t)$
 that satisfies the following growth conditions
$$
l(t)\prec t\  \text{ and }   \    (l(t))^{k+1} a^{(k+1)}(t)\prec 1\prec (l(t))^k a^{(k)}(t).
$$
As in the proof of Lemma~\ref{L:basicequi},  one checks that the  existence of such a function is guaranteed by our assumption, concluding the proof.
\end{proof}
Notice that by  Lemma~\ref{L:cvb} the previous result applies to every  function $a\in \H$ that has  polynomial growth
and satisfies  $|a(t)-p(t)|\succ \log{t}$ for every $p\in \R[t]$.
 In order to
deal with the remaining cases of Proposition~\ref{P:B} we  need one more lemma. Its proof follows the same strategy as in the
 proof of Lemma~\ref{L:uio}, so in order to avoid unnecessary repetition our argument will be rather sketchy.
\begin{lemma}\label{L:uio2}
Let $a\in \H$ satisfy  $a(t)=p(t)+e(t)$, where  $e(t)\ll \log{t}$, and $p\in\R[t]$ is not of  the form  $c q(t)+d$
 with $c,d\in \R$ and  $q\in \Z[t]$.  Let $X=G/\Gamma$ be a nilmanifold, with $G$
connected and simply connected, and suppose that  $b\in G$ acts ergodically on $X$.

Then
the sequence $(b^{a(n)}\Gamma)_{n\in\N}$ is  equidistributed in $X$.
\end{lemma}
\begin{remark}
Our argument  can easily be adapted to cover  every function  $e\in \H$ that satisfies $e(t)\prec t$, but
 the case treated suffices for our purposes.
\end{remark}
\begin{proof}
Arguing as in Lemma~\ref{L:uio}, it suffices to show that for every $F\in C(X)$ with zero integral, the averages
\begin{equation}\label{E:bhn1''}
 \E_{N\leq n\leq N+\sqrt{N}} F(b^{a(n)}\Gamma)
\end{equation}
converge to zero as $N\to \infty$.

  Using Lemma~\ref{L:properties} we conclude that the function $|e'(t)|$ is decreasing and  $e'(t)\prec 1/t^{1-\varepsilon}$
  for every $\varepsilon>0$. Using the mean value theorem we conclude that for $n\in [1,\sqrt{N}]$ we have
  $$
  e(N+n)=e(N)+o_{N\to \infty}(1)
  $$
  and as a result the averages in
  \eqref{E:bhn1''}
  are equal to
$$
 \E_{1\leq n\leq \sqrt{N} } F\Big(b^{p(N+n)+e(N)}
\Gamma\Big)+o_{N\to\infty}(1).
$$
Hence, our proof will be complete if we show  that
 for every $\delta>0$, and every large $N$, the finite sequence
$(b^{p(N+n)+e(N)}\Gamma)_{1\leq n\leq \sqrt{N}}$  is $\delta$-equidistributed in $X$.
Suppose that this is not the case.
We are going to  use Theorem~\ref{T:GT} to derive a contradiction.
The key property  to be used is  that for every non-zero real number $\beta$ the polynomial $\beta p(t)$ has at least
one non-constant coefficient irrational.  Arguing as in  Lemma~\ref{L:uio}, we deduce
 that there exists a constant $M$ (depending only on $\delta$, $X$
and the degree of $p$), and a finite set $B$ of  irrational numbers, such that for infinitely many positive integers
 $N$ we have
$$
\sqrt{N}\norm{\beta}\leq M
$$
for some $\beta\in B$. This is a contradiction and the proof is complete.
\end{proof}

Combining the last two lemmas it is now easy to prove Proposition~\ref{P:B}.
\begin{proof}[Proof of Proposition~\ref{P:B} (Conclusion of proof of Theorem~\ref{T:B})]
Using Lemmas~\ref{L:cvb} and \ref{L:uio} we cover  the case where $|a(t)-p(t)|\succ \log{t}$
for every $p\in \R[t]$.
The remaining cases are covered by  Lemma~\ref{L:uio2}.
\end{proof}

Next we prove Theorem~\ref{T:A}. It is   a direct consequence of Theorem~\ref{T:B} and  the following lemma:
\begin{lemma}\label{L:polynomial}
Let $a\in \H$ satisfy $a(t)-p(t)\to 0$ for some $p\in \R[t]$.

Then the sequences $(a(n))_{n\in\N}$
and $([a(n)])_{n\in\N}$ are pointwise good for nilsystems.
\end{lemma}
\begin{proof}
Let $X=G/\Gamma$ be a nilmanifold, with $G$ connected and simply connected, and   $b\in G$.
The sequence $(a(n))_{n\in\N}$ is pointwise good for nilsystems if and only if the same holds for the
sequence $(p(n))_{n\in\N}$. Let $p(t)=c_0+c_1t+\cdots +c_kt^k$ for some non-negative integer $k$ and $c_i\in \R$.
Since  $b^{p(n)}=b_0 \cdot b_1^n \cdot \ldots \cdot b_k^{n^k}$, where $b_i=b^{c_i}$, we have that
$(b^{p(n)})_{n\in \N}$ is a polynomial sequence in $G$. It follows by Theorem~\ref{T:L} that the sequence $(p(n))_{n\in\N}$
is pointwise good for nilsystems.

Next we deal with the sequence  $([a(n)])_{n\in\N}$.
Suppose first that  $p(t)-p(0)\in \Q[t]$. Then $p(t)=\frac{1}{r}\tilde{p}(t)+c$ for some $r\in \N$, $c\in\R$, and $\tilde{p}\in\Z[t]$
with $p(0)=0$.
For $i=0,\ldots,r-1$ we have  $[a(rn+i)]=q_i(n)$ for some $q_i\in \Z[t]$.
Using this, the result follows from Theorem~\ref{T:L}.

It remains to deal with the case where  the polynomial $p$ has an irrational non-constant coefficient.
We are going to use a strategy similar to the
one used  in the proof of
 Lemma~\ref{L:reduction}.
Let
$$
\tilde{X}=\tilde{G}/\tilde{\Gamma} \ \text{ where }\
 \tilde{G}=\R\times G, \ \tilde{\Gamma}=\Z\times\Gamma, \ \text{
and }  \ \tilde{b}=(1,b).
$$
 Given   $F\in C(X)$ we define $\tilde{F}\colon \tilde{X}\to \C$  by
  $$
  \tilde{F}(t\Z,g\Gamma)=F(b^{-\{t\}}g\Gamma).
  $$
  (We caution the reader that $\tilde{F}$ may not be continuous).
Notice that
  $$
F(b^{[a(n)]}\Gamma)=F(b^{-\{a(n)\}}b^{a(n)}\Gamma)=\tilde{F}(\tilde{b}^{a(n)}\tilde{\Gamma}),
$$
and as a result it suffices to show that the averages
\begin{equation}\label{E:tilde}
\E_{1\leq n\leq N} \tilde{F}(\tilde{b}^{a(n)}\tilde{\Gamma})
\end{equation}
converge as $N\to\infty$. We verify this as follows.  For  $\delta>0$ (and smaller than $1/2$) there exist functions  $\tilde{F}_\delta\in C(\tilde{X})$ that agree with $\tilde{F}$ on $\tilde{X}_\delta=I_\delta \times X$, where
$I_\delta=\big\{t\Z\colon \norm{t}\geq \delta \big\}$, and are uniformly bounded by $2\norm{F}_\infty$. Since the polynomial $p$ has a non-constant irrational coefficient,   the sequence $(p(n)\Z)_{n\in\N}$ is equidistributed in $\T$, and as a result     $\tilde{b}^{p(n)}\tilde{\Gamma}\in \tilde{X}_\delta$ for a set of  $n\in\N$ with density $1-2\delta$. It follows that
\begin{equation}\label{E:approx}
\limsup_{N\to\infty} \E_{1\leq n\leq N}|\tilde{F}(\tilde{b}^{a(n)}\tilde{\Gamma})- \tilde{F}_\delta(\tilde{b}^{a(n)}\tilde{\Gamma})|\leq 4\norm{F}_\infty\delta.
\end{equation}
  As shown in the first part of our proof, the sequence $(a(n))_{n\in\N}$ is pointwise good for nilsystems, hence  the averages \eqref{E:tilde} converge when one uses the functions $\tilde{F}_\delta$ in place of the function $\tilde{F}$.  Using this and \eqref{E:approx}, we deduce that the
 averages in \eqref{E:tilde} form a Cauchy sequence, and hence they converge as $N\to \infty$. This proves that
the sequence  $([a(n)])_{n\in\N}$ is  pointwise good for nilsystems and completes the proof.
\end{proof}
\begin{proof}[Proof of Theorem~\ref{T:A}]
The sufficiency of the conditions follows immediately  from Theorem~\ref{T:B} and Lemma~\ref{L:polynomial},
with the exception of the case where $|a(t)-t/m|\ll \log{t}$ for some non-zero integer $m$.  As noticed in \cite{BKQW} (proof of Theorem~3.3), this last case
is easily reduced to the case $a(t)=t$. In this particular instance the result is well known  (e.g. \cite{L2}).

The necessity of the conditions can be seen by working with rational rotations on the circle,
for the details see \cite{BKQW}.
\end{proof}

\section{Several nil-orbits and Hardy sequences}\label{S:multiple}
In this section we shall prove Theorem~\ref{T:several}. A crucial part of our argument will
 be  different than the one used to prove of Theorem~\ref{T:B}, so we find it
 instructive to  start with a model equidistribution problem that illustrates the key technical difference.
\subsection{A model equidistribution problem}\label{SS:model2}
We shall give yet another proof of the following  special case of Theorem~\ref{T:Bos}:

``If
$a\in\H$ satisfies $(t \log t)  \prec a(t)\prec t^2$, then
the sequence $(a(n))_{n\in\N}$ is equidistributed in $\T$."

We  shall take the following fact for granted:

``If $a\in\H$ satisfies $ \log t \prec a(t)\prec t$, then
the sequence $(a(n))_{n\in\N}$ is equidistributed in $\T$."

 So suppose that  $a\in\H$ satisfies $(t \log t)  \prec a(t)\prec t^2$.
It suffices to show that for every non-zero $k\in\Z$ we have
\begin{equation}\label{EE:j}
\lim_{N\to\infty} \E_{1\leq n\leq N} e(ka(n))=0.
\end{equation}
For convenience we assume that $k=1$.
For every fixed  $R\in \N$ we have
\begin{equation}\label{E:Rn+r}
\E_{1\leq n\leq  RN} e(a(n))=
\E_{1\leq n\leq N} \big(\E_{1\leq r\leq R} \ e(a(Rn+r))\big)+o_{N\to\infty}(1).
\end{equation}
For $n=1,2,\ldots$, we use the Taylor expansion of
 $a(t)$ around the point $t=Rn$. Since $a''(t)\to 0$ (by Lemma~\ref{L:properties}),
we get for $ r\in [1, R]$ that
$$
a(Rn+r)=a(Rn)+ra'(Rn)+o_{n\to\infty;R}(1).
$$
It follows that  the averages in \eqref{E:Rn+r} are equal to
\begin{equation}\label{E:Rn}
\E_{1\leq n\leq N} A_{R,n}+o_{N\to\infty;R}(1), \text{ where }
A_{R,n}=\E_{1\leq r\leq R}\ e(a(Rn)+ra'(Rn)).
\end{equation}
For fixed $\varepsilon>0$ we split  the averages  $\E_{1\leq n\leq N} |A_{R,n}|$  as follows
$$
\E_{1\leq n\leq N} ({\bf 1}_{n\colon \norm{a'(Rn)}\leq \varepsilon}\cdot  |A_{R,n}|)
+ \E_{1\leq n\leq  N} ({\bf 1}_{n\colon \norm{a'(Rn)}> \varepsilon}\cdot  |A_{R,n}|)=\Sigma_{1,R,N,\varepsilon}+\Sigma_{2,R,N,\varepsilon}.
$$

We estimate $\Sigma_{1,R,N,\varepsilon}$. By Lemma~\ref{L:properties} we have that
 $\log t\prec a'(Rt)\prec t$. It follows that the sequence $(a'(Rn)\Z)_{n\in\N}$
is equidistributed in $\T$, and as a consequence
$$
\frac{|1\leq n\leq N\colon \norm{a'(Rn)}\leq \varepsilon|}{N}=2\varepsilon+o_{N\to\infty;R}(1).
$$
Therefore,  $\Sigma_{1,R,N,\varepsilon}\leq 2\varepsilon+o_{N\to\infty;R}(1)$.

We estimate $\Sigma_{2,R,N,\varepsilon}$.  We have
$$
|A_{R,n}|=|\E_{1\leq r\leq R}\ e(ra'(Rn))|.
$$
We estimate the geometric series in the standard fashion; computing the sum and  using the estimate $|\sin{\pi t}| \geq 2\norm{t}$, we find  that (whenever $a'(Rn)$ is not an integer)
$$
|A_{R,n}|\leq \frac{1}{2R\norm{a'(Rn)}}.
$$
 It follows that $\Sigma_{2,R,N,\varepsilon}\leq 1/(2R\varepsilon)$.

Combining the estimates for $\Sigma_{1,R,N,\varepsilon}$ and $\Sigma_{2,R,N,\varepsilon}$ we get
$$
\E_{1\leq n\leq N} |A_{R,n}|\leq 2\varepsilon+\frac{1}{2R\varepsilon}+o_{N\to\infty;R}(1).
$$
Letting first $N\to\infty$, then $R\to \infty$, and then $\varepsilon \to 0$, we get
 $$
\lim_{R\to\infty}\lim_{N\to\infty}\E_{1\leq n\leq N} |A_{R,n}|=0.
$$
  As explained before, this implies \eqref{EE:j}
 and completes the proof.

\subsection{A reduction}
As was the case with  the proof of Theorem~\ref{T:B}, we start with some
initial maneuvers that enable us to reduce Theorem~\ref{T:several} to a more convenient  statement.
Since this step can
be completed with straightforward modifications of the arguments used in Section~\ref{SS:reduction},
we omit the proofs.

First notice that in order to prove Theorem~\ref{T:several} we can assume that $X_1=\dots=X_{\ell}=X$.
 Indeed, consider the nilmanifold $\tilde{X}=X_1\times\cdots\times X_{\ell}$.
Then $\tilde{X}=\tilde{G}/\tilde{\Gamma}$, where $\tilde{G}=G_1\times\cdots\times G_{\ell}$  is connected
and simply connected, $\tilde{\Gamma}=\Gamma_1\times\cdots\times \Gamma_{\ell}$  is a discrete cocompact
subgroup of $\tilde{G}$,
each $b_i$  can be thought of as an element of $\tilde{G}$, and each  $x_i$ as an element of $\tilde{X}$.

\begin{lemma}\label{L:reduction'}
Let $(a_1(n))_{n\in\N}, \ldots, (a_{\ell}(n))_{n\in\N}$ be  sequences of real numbers.
Suppose that   for every nilmanifold $X=G/\Gamma$, with
 $G$ connected and simply connected,  and every $b_1,\ldots,b_{\ell}\in G$,
the sequence $$
(b^{a_1(n)}_1\Gamma,\ldots,b^{a_{\ell}(n)}_{\ell}\Gamma)_{n\in\N}
$$ is equidistributed in the
nilmanifold
$\overline{(b^s_1\Gamma)}_{s\in \R}\times\cdots\times \overline{(b^s_{\ell}\Gamma)}_{s\in \R}$.

Then for every nilmanifold $X=G/\Gamma$,  every   $b_1,\ldots,b_{\ell}\in G$, and $x_1,\ldots,x_{\ell}\in X$,
the sequence
$$
(b^{[a_1(n)]}_1x_1,\ldots,b^{[a_{\ell}(n)]}_{\ell}x_{\ell})_{n\in\N}
$$ is equidistributed in the nilmanifold
 $\overline{(b^n_1x_1)}_{n\in \N}\times\cdots\times \overline{(b^n_{\ell}x_{\ell})}_{n\in \N}$.
\end{lemma}

The previous lemma shows that part $(ii)$ of Theorem~\ref{T:several} follows from part $(i)$.

\begin{lemma}\label{L:ergodic'}
 Let  $X=G/\Gamma$ be a nilmanifold with $G$
connected and simply connected.

Then for every
 $b_1,\ldots,b_{\ell}\in G$ there exists $s_0\in \R$ such that for $i=1,\ldots,\ell$  the element
  $b^{s_0}_i$ acts ergodically
 on the nilmanifold $ \overline{(b^s_i\Gamma)}_{s\in \R}$.
\end{lemma}

Using Lemmas~\ref{L:reduction'} and \ref{L:ergodic'}, we see as in section Section~\ref{SS:reduction},
that Theorem~\ref{T:several} reduces to
proving the following result:
\begin{proposition}\label{P:B'}
  Suppose that the functions $a_1(t),\ldots,a_{\ell}(t)$ belong to the same Hardy field, have different growth rates, and satisfy $t^k\log{t}\prec a_i(t)\prec t^{k+1}$ for some $k=k_i\in \N$.

Then given nilmanifolds $X_i=G_i/\Gamma_i$, $i=1,\ldots, \ell$, with
$G_i$ connected and simply connected, and elements
$b_1,\ldots,b_{\ell}\in G_i$ acting ergodically on $X_i$, the
sequence
$$
({b}^{a_1(n)}_1\Gamma_1,\ldots,{b}^{a_{\ell}(n)}_{\ell}\Gamma_\ell)_{n\in\N}
$$ is equidistributed in the nilmanifold $ X_1\times \cdots \times X_\ell$.
\end{proposition}

\subsection{Proof of Proposition~\ref{P:B'}}
Since there is a key technical difference in the proofs of  Proposition~\ref{P:B'}
and Proposition~\ref{P:B}, we are going to give all the details.
We are going to  adapt the proof  technique of the model equidistribution result of Section~\ref{SS:model2}
to our particular non-Abelian setup.

\begin{proof}[Proof of Proposition~\ref{P:B'}]
For convenience of exposition we assume that $X_1=\cdots=X_\ell=X$,
the proof in the general case is similar.  Let $F\in C(X^{\ell})$
with zero integral. We want to show that
\begin{equation}\label{EE:main}
\lim_{N\to\infty} \E_{1\leq n\leq N} F
({b}^{a_1(n)}_1\Gamma,\ldots,{b}^{a_{\ell}(n)}_{\ell}\Gamma)=0.
\end{equation}
For every fixed  $R\in \N$ we have
\begin{equation}\label{E:Rn+r''}
\E_{1\leq n\leq  RN} F({b}^{a_1(n)}_1\Gamma,\ldots,{b}^{a_{\ell}(n)}_{\ell}\Gamma)=
\E_{1\leq n\leq N} \big(\E_{1\leq r\leq R} \ F({b}^{a_1(nR+r)}_1\Gamma,\ldots,{b}^{a_{\ell}(nR+r)}_{\ell}\Gamma)\big)
+o_{N\to\infty;R}(1).
\end{equation}
For $n=1,2,\ldots$, we use the Taylor  expansion of the functions
$a_i(t)$ around the point $t=Rn$. Since $t^{k_i}\log{t}\prec a_i(t)\prec t^{k_i+1}$ for some $k_i\in\N$, Lemma~\ref{L:properties} gives that $a^{(k_i+1)}_i(t)\to 0$. Hence,
for $r\in [1, R]$ we have that
$$
a_i(Rn+r)=p_{i,R,n}(r)+o_{n\to\infty;R}(1),
$$
where
\begin{equation}\label{E:pirn}
p_{i,R,n}(r)=a_i(Rn)+ra'_i(Rn)+\cdots+\frac{r^{k_i}}{k_i!}a^{(k_i)}_i(Rn).
\end{equation}
   It follows that  the averages  in \eqref{E:Rn+r''} are equal to
\begin{equation}\label{E:Rn''}
\E_{1\leq n\leq N} A_{R,n} +o_{N\to\infty;R}(1), \text{ where }
A_{R,n}=\E_{1\leq r\leq R}\ F(b^{p_{1,R,n}(r)}_1\Gamma,\ldots,b^{p_{\ell,R,n}(r)}_{\ell}\Gamma).
\end{equation}
Our objective now is to show that for every $\delta>0$, for all large values of  $R$, the finite sequence
$(b^{p_{1,R,n}(r)}_1,\ldots,b^{p_{\ell,R,n}(r)}_\ell)_{1\leq r\leq R}$ is $\delta$-equidistributed in $X^\ell$
for most values of $n$. This will enable us to show that the averages in \eqref{E:Rn''} converge
to zero as $N\to \infty$.

So let $\delta>0$.  As in the proof of Proposition~\ref{P:B} we verify that for fixed $R,n\in\N$ the sequence
 $(b^{p_{1,R,n}(r)}_1,\ldots,b^{p_{\ell,R,n}(r)}_\ell)_{r\in\N}$ is a polynomial sequence in $G^k$.
 Since $X^{\ell}=G^{\ell}/\Gamma^{\ell}$, and $G^{\ell}$ is  connected and
simply connected, we can apply  Theorem~\ref{T:GT} (for small  $\delta$). We get that
if the finite sequence $(b^{p_{1,R,n}(r)}_1\Gamma,\ldots,b^{p_{\ell,R,n}(r)}_{\ell}\Gamma)_{1\leq r\leq R}$ is not $\delta$-equidistributed in $X^\ell$, then there exists
 a constant $M$ (depending only on $\delta$, $X$,  and  the $k_i$'s), and a non-trivial  horizontal character $\chi$ of $X^\ell$, with $\norm{\chi}\leq M$, and such that
\begin{equation}\label{E:C_R'}
\norm{\chi(b_1^{p_{1,R,n}(r)},\ldots, b_{\ell}^{p_{\ell,R,n}(r)})}_{C^\infty[R]}\leq M.
\end{equation}
For $i=1,\ldots,\ell$, let $\pi(b_i)=(\beta_{i,1}\Z,\ldots,\beta_{i,s}\Z)$, where $\beta_{i,j}\in \R$,  be the projection of $b_i$ on the horizontal torus $\T^s$ of $X$
(notice that
$s$ is bounded by the dimension of $X$).
Since each $b_i$ acts ergodically on $X$,
 the set of real numbers $\{1,\beta_{i,1},\ldots,\beta_{i,s}\}$  is rationally independent for
 $i=1,\ldots,\ell$.
    For $t\in \R$ we have $\pi(b_i^t)=(t\tilde{\beta}_{i,1}\Z,\ldots, t\tilde{\beta}_{i,s}\Z)$ for some
 $\tilde{\beta}_{i,j}\in \R$ with  $\tilde{\beta}_{i,j}\Z=\beta_{i,j} \Z$.
As a consequence
\begin{equation}\label{E:bve}
\chi(b_1^{p_{1,R,n}(r)},\ldots, b_{\ell}^{p_{\ell,R,n}(r)})=e\Big(\sum_{i=1}^s \big(p_{i,R,n}(n)\sum_{j=1}^s l_{i,j}\tilde{\beta}_{i,j}\big)\Big)
\end{equation}
for some $l_{i,j}\in\Z$ with $|l_{i,j}|\leq M$.
 Let $k_{\text{min}}=\min\{k_1,\ldots,k_{\ell}\}$
 and $k_{\text{max}}=\max\{k_1,\ldots,k_{\ell}\}$. It follows from \eqref{E:pirn}, \eqref{E:bve},
  and the definition of
  $\norm{\cdot}_{C^\infty[R]}$ (see \eqref{E:norms}), that there exists $k\in \N$ with
 $k_{\text{min}}\leq k\leq k_{\text{max}}$ such that
$$
\norm{\chi(b_1^{p_{1,R,n}(r)},\ldots, b_{\ell}^{p_{\ell,R,n}(r)})}_{C^\infty[R]}\geq R^k
\Big\|\sum_{i\in I} a_i^{(k)}(Rn)\beta_i\Big\|,
$$
where the sum ranges over those $i\in\{1,\ldots,\ell\}$ that satisfy $k_i=k$, and
the $\beta_i$'s are  real numbers, not all of them zero
(we used here that $\chi$ is non-trivial and  the rational independence of the $\tilde{\beta}_{i,j}$'s),
that belong to the finite set
$$
B=\bigcup_{i=1}^\ell\Big\{\frac{1}{k!}\sum_{j=1}^s l_{i,j}\tilde{\beta}_{i,j}\colon |l_{i,j}|\leq M\Big\}.
$$
 Combining this estimate with \eqref{E:C_R'} gives
\begin{equation}\label{E:a'(Rn)'}
\Big\|\sum_{i\in I} a_i^{(k)}(Rn)\beta_i\Big\|\leq \frac{M}{R^k}
\end{equation}
for some $\beta_i\in B$.

We are now ready to estimate the average
$\E_{1\leq n\leq N} |A_{R,n}|$.
Given $\varepsilon>0$
we split  it  as follows
$$
\E_{1\leq n\leq N} |A_{R,n}|=\E_{1\leq n\leq N} ({\bf 1}_{S_{1,R,\varepsilon}}(n)\cdot  |A_{R,n}|)
+ \E_{1\leq n\leq  N} ({\bf 1}_{S_{2,R,\varepsilon}}(n)\cdot  |A_{R,n}|)=\Sigma_{1,R,N,\varepsilon}+\Sigma_{2,R,N,\varepsilon},
$$
where
$$
S_{1,R,\varepsilon}=\Big\{n\in \N\colon
\Big\| \sum_{i\in I} a_i^{(k)}(Rx)\beta_i\Big\|\leq \varepsilon \text{ for some }
\beta_i\in B, \text{ not all of them } 0\Big\}, \ S_{2,R,\varepsilon}=\N\setminus S_{1,R,\varepsilon}.
$$

We estimate $\Sigma_{1,R,N,\varepsilon}$. Using Lemma~\ref{L:properties} and our assumptions, we
conclude that $\log{t}\prec a^{(k)}_i(t)\prec t$ for $i\in I$. Furthermore,
since the functions $a_i(t)$ for $i\in I$  have different growth rates and belong to the same Hardy field,
 we deduce that
the functions $a^{(k)}_i(t)$ for $i\in I$
have different growth rates. It follows that
$$
\log t\prec b_R(t)=\sum_{i\in I} a_i^{(k)}(Rt) \beta_i\prec t.
$$
 Since $b_R\in \H$  and $\log{t} \prec b_R(t)\prec t$, we get (e.g. using Theorem~\ref{T:Bos}) that  for every $R\in \N$ the sequence $(b_R(n)\Z)_{n\in\N}$ is equidistributed in $\T$. Hence,
$$
\frac{|1\leq n\leq N\colon \norm{b_R(n)}\leq \varepsilon|}{N}=2\varepsilon+o_{N\to\infty;R}(1).
$$
It follows that
  $$
  |\Sigma_{1,R,N,\varepsilon}|\leq 2\norm{F}_\infty \varepsilon+o_{N\to\infty;R}(1).
  $$

We estimate $\Sigma_{2,R,N,\varepsilon}$. Notice that for $n\in S_{2,R,\varepsilon}$ we have
$\norm{\sum_{i\in I} a_i^{(k)}(Rn)\beta_i}\geq\varepsilon$. As a result,
 if $R$ is large enough, then  \eqref{E:a'(Rn)'} fails, and as a consequence the finite sequence
$(b_1^{p_{1,R,n}(r)}x_1,\ldots, b_{\ell}^{p_{\ell,R,n}(r)}x_{\ell})_{1\leq r\leq R}$ is  $\delta$-equidistributed in $X$. Hence, if $R$ is large enough, then
$|A_{R,n}|\leq \delta$ for every  $n\in S_{2,R,\varepsilon}$.
Therefore, for every $N\in\N$ we have
$$
|\Sigma_{2,R,N,\varepsilon}|\leq \delta +o_{R\to\infty}(1).
$$

Putting the previous estimates together we find
$$
\E_{1\leq n\leq N} |A_{R,n}|\leq 2\norm{F}_\infty \varepsilon+\delta+ o_{N\to\infty;R}(1) +o_{R\to\infty}(1).
$$
Letting $N\to\infty$, then $R\to \infty$, and then $\varepsilon,\delta \to 0$, we deduce that
$$
\lim_{R\to\infty}\limsup_{N\to\infty}\E_{1\leq n\leq N} |A_{R,n}|=0,
$$
or equivalently that
$$
\lim_{R\to\infty}\limsup_{N\to\infty}\E_{1\leq n\leq N}\Big|
\E_{1\leq r\leq
R}F(b_1^{[a_1(Rn+r)]}\Gamma,\ldots,b_\ell^{[a_\ell(Rn+r)]}\Gamma)\Big|=0.
$$
Combining this with \eqref{E:Rn+r''} gives \eqref{EE:main}, completing the proof.
\end{proof}

\section{Random  sequences of sub-exponential growth}\label{S:random}
In this section we shall prove Theorem~\ref{T:random}. In what follows,
when we introduce   a nilpotent Lie group $G$ or a nilmanifold $X$,  we assume that it comes equipped
with a Mal'cev basis
and the corresponding (right invariant) metric $d_G$ or $d_X$  that was introduced in \cite{GT}.
When there is no danger of confusion we are going to denote $d_G$ or $d_X$ with $d$.
We denote by $B_M$ the ball in $G$ of radius $M$, that is
  $B_M=\{g\in G\colon d(g,\text{id}_G)\leq M\}.
  $


\subsection{A reduction}

We start with  some initial maneuvers that will allow us  to reduce Theorem~\ref{T:random} to
a more convenient statement.

We remind the reader of our setup. We are given
a sequence  $(X_n(\omega))_{n\in\N}$ of $0-1$ valued independent random
 variables  with  $P(\{\omega\in \Omega\colon X_n(\omega)=1\})=\sigma_n$, where
 $(\sigma_n)_{n\in\N}$ is a decreasing sequence of real numbers
 satisfying $\lim_{n\to \infty} n\sigma_n=\infty$.
Our objective is to show that
almost surely the averages
\begin{equation}\label{E:Av0}
\frac{1}{N}\sum_{n=1}^NF(b^{a_n(\omega)}\Gamma)
\end{equation}
converge as $N\to \infty$ for every nilmanifold $X=G/\Gamma$, function  $F\in C(X)$, and element $b\in G$.

We  caution the reader  that  the set of probability $1$ for which
the averages \eqref{E:Av0} converge has
to be independent of the nilmanifold $X=G/\Gamma$, the function $F\in C(X)$, and the
element  $b\in G$. On the other hand, since
up to isomorphism there exist  countably many
nilmanifolds $X$ (see for example \cite{CG}), and since the space $C(X)$ is separable, it suffices to
prove that for every fixed nilmanifold $X=G/\Gamma$ and $F\in C(X)$  the averages \eqref{E:Av0}
converge almost surely for every  $b\in G$. Furthermore, since $G$ is a countable union
of balls, it suffices to verify the previous statement  with $B_M$ in place of $G$ for every $M>0$.

Next notice that instead of working with the averages \eqref{E:Av0}, it suffices to work with  the averages
$$
\frac{1}{A(N,\omega)}\sum_{n=1}^N X_n(\omega)\ \! F(b^{n}\Gamma)
$$
where $A(N,\omega)=|\{n\in \{1,\ldots, N\} \colon  X_n(\omega)=1\}|$.
Since the expectation of $X_n$ is $\sigma_n$,  by the strong law of large numbers we
almost surely have that  $A(N,\omega)/w(N)\to 1,$ where $w(N)=\sum_{n=1}^N\sigma_n$.
It  therefore  suffices to work with the averages
$$
\frac{1}{w(N)}\sum_{n=1}^N X_n(\omega)\ \! F(b^{n}\Gamma).
$$
We shall establish convergence of these  averages
by comparing them with the averages
$$
\frac{1}{w(N)}\sum_{n=1}^N \sigma_n\ \! F(b^{n}\Gamma).
$$
Notice that these last averages can be compared with the averages
$$
\frac{1}{N}\sum_{n=1}^N F(b^{n}\Gamma)
$$
which, as we have mentioned repeatedly before,  are known to be convergent.

Up to this point we have reduced matters to showing that for every nilmanifold $X=G/\Gamma$,  $F\in C(X)$, and $M>0$,
we almost surely have
\begin{equation}\label{E:nilsequence}
\lim_{N\to \infty} \frac{1}{w(N)}\sum_{n=1}^N (X_n(\omega)-\sigma_n)\ \!F(b^n\Gamma)=0
\end{equation}
 for every $b\in B_M$.

Next we  show that we can impose a few extra assumptions on the nilmanifold $X$, and the function $F\in C(X)$.
 Using the lifting argument of Section~\ref{SS:nil} we see that every sequence
 $(F(b^n\Gamma))_{n\in\N}$ can be represented in the form
  $(\tilde{F}(\tilde{b}^n\tilde{\Gamma}))_{n\in\N}$
 for some nilmanifold $\tilde{X}=\tilde{G}/\tilde{\Gamma}$, with $\tilde{G}$
 connected and simply connected, $\tilde{F}\in C(\tilde{X})$, and $\tilde{b}\in \tilde{G}$.
 Therefore,  when proving \eqref{E:nilsequence}
 we can assume that the nilmanifold $X$ has the form $G/\Gamma$, where the  group
 $G$ is connected and simply connected. Furthermore, since the set $\text{Lip}(X)$,  of Lipschitz functions
 $F\colon X\to \C$,  is
 dense in $C(X)$ in the uniform topology,
 an easy approximation argument shows that it suffices to prove
 \eqref{E:nilsequence} for $F\in \text{Lip}(X)$.

Summarizing, we have reduced Theorem~\ref{T:random} to proving:
\begin{theorem}\label{T:nilsequence}
Let $(X_n(\omega))_{n\in\N}$ be a  sequence of $0-1$ valued independent random
 variables  with  $P(\{\omega\in \Omega\colon X_n(\omega)=1\})=\sigma_n$, where
 $(\sigma_n)_{n\in\N}$ is a decreasing sequence of real numbers
 satisfying $\lim_{n\to \infty} n\sigma_n=\infty$. Let $X=G/\Gamma$ be a nilmanifold,
  with $G$ connected and simply connected,   $F\in \text{Lip}(X)$,
   and $M>0$.

 Then almost surely we have
$$
\lim_{N\to\infty}\max_{b\in B_M} \left|\frac{1}{w(N)}\sum_{n=1}^N (X_n(\omega)-\sigma_n)\ \!F(b^n\Gamma)
\right| \quad=0
$$
where $w(N)=\sum_{n=1}^N\sigma_n$.
\end{theorem}
\begin{remark}
We shall not use the fact that the  convergence to zero is uniform; only the independence of the set of full measure
on the set $B_M$ will be used.
\end{remark}
To prove Theorem~\ref{T:nilsequence} we are going to extend an argument  used
by Bourgain in \cite{Bo} (where the case  $X=\T$ was covered).
A more detailed version of this argument can be found in \cite{RW}.
  Since several steps of \cite{RW} carry over verbatim  to our case we are only going to spell
 out the details of the genuinely  new steps.

\subsection{A key ingredient}
 In this subsection we shall prove the following key result:
\begin{proposition} \label{P:random}
 Let $X=G/\Gamma$ be a  nilmanifold, with $G$ connected and simply connected,
    and $M$ be a positive real number.

 Then  there exists $k=k(G,M)\in \N$ with the following property:
 for every $N\in\N$, there exists  $B_{N,M}\subset B_M$ with  $|B_{N,M}|=N^k$,
 and such that  for every  $F\in \text{Lip}(X)$ with $\norm{F}_{\text{Lip}(X)}\leq 1$,
  and sequence  of real numbers $(c_n)_{n\in\N}$ with norm bounded by $1$ , we have
\begin{equation}\label{E:approximation}
\max_{b\in B_M}\Big|\sum_{n=1}^N c_n\ \! F(b^n\Gamma)\Big|=
\max_{b\in B_{N,M}}\Big|\sum_{n=1}^N c_n\ \! F(b^n\Gamma)\Big|+o_{N\to\infty}(1).
\end{equation}
\end{proposition}

The proof of this result ultimately relies on the fact that
multiplication on a nilpotent Lie group
is  given by polynomial mappings.
To make this precise we shall use a convenient coordinate system,
the proof of its existence can be found  in \cite{GT} (for example).

 For every connected and simply connected Lie group $G$
there exist a non-negative integer $m$ (we call $m$ the  \emph{dimension} of $G$)
and a continuous isomorphism
$\phi$ from $(G,\cdot)$ to $(\R^m,\cdot)$ with multiplication defined as follows:
If   $u=(u_1,\ldots, u_m)$ and $v=(v_1,\ldots,v_m)$, then for $i=1,\ldots,m$ the $i$-th coordinate of
 $u\cdot v$ has the form
$$
u_i+v_i+P_i(u_1,\ldots,u_{i-1},v_1,\ldots, v_{i-1})
 $$
 where $P_i\colon \R^{i-1}\times \R^{i-1} \to \R$ is a polynomial of degree at most $i$.
  It follows that
 the $i$-th coordinate of $u^n$ has the form
 $$
 nu_i+Q_i(u_1, \ldots, u_{i-1},n)
 $$
where $Q_i\colon \R^{i-1}\times \R\to \R$ is a polynomial.




 We shall use the following result (Lemma A.4 in \cite{GT}):
 \begin{lemma}[{\bf Green \& Tao \cite{GT}}]\label{L:GT}
 Let $G$ be a connected and simply connected nilpotent Lie group of dimension $m$.

 Then  there  exists $k=k(G)\in \N$ such that for every $K>1$ we have
$$
   K^{-k}|u-v|\leq d(g,h) \leq K^k|u-v|,
  $$
for every $g,h\in G$, and  $u=\phi(g), v=\phi(h)\in \R^m$ that satisfy  $|u|$, $|v|\leq K$,
where $|\cdot|$ denotes the sup-norm in $\R^m$.
 \end{lemma}
 Using this, we are going to  show:
 \begin{lemma}\label{L:g^n}
 Let $G$ be a connected and simply connected  nilpotent Lie group.

Then   there exists $k=k(G)\in \N$ such that for every $M>0$ we have
$$
d(g^n,h^n)\ll_{G,M}  n^k d(g,h)
$$
for every $n\in \N$ and  $g, h\in B_M$.
 \end{lemma}
\begin{proof}
We first establish the corresponding estimate in  ``coordinates''. Suppose that the
dimension of $G$ is $m$.
Let $\phi(g)=u=(u_1,\ldots, u_m)$ and $\phi(h)=v=(v_1,\ldots,v_m)$   satisfy
$|u|, |v|\leq K$.
Using the multiplication formula in local coordinates we deduce that
 $$
 |(u^n)_i-(v^n)_i|\leq \sum_{j=1}^i |(u_j-v_j)|\ \! |R_j(u_1,\ldots,u_{j-1},v_1,\ldots, v_{j-1}, n) |
 $$
 for some polynomials $R_j\colon \R^{j-1} \times \R^{j-1}\times \R \to \R$ of degree of degree depending only on
 $G$.
If we consider $R_j$   as a polynomial of a single variable $n$,
then  its coefficients depend polynomially on  the parameters $u_i, v_i$
 (which are bounded by $K$) and the structure constants of the Mal'cev basis of  $G$. Hence,
 $
 |R_j(u,v,n)|\ll_{G,K}  n^{l_j}
 $
 for some $l_j=l_j(G)\in \N$. It follows that
 $$
 |(u^n)_i-(v^n)_i|\ll_{G,K}  n^{k_1} \sum_{j=1}^i |(u_j-v_j)|
 $$
 for some $k_1=k_1(G)$. As a consequence
 \begin{equation}\label{E:local}
 |u^n-v^n|\ll_{G,K} n^{k_1} |u-v|
 \end{equation}
 for some $k_1=k_1(G)$. 

To finish the proof, we use \eqref{E:local} to deduce an analogous estimate for elements of $G$ with the metric $d$.
We argue as follows.
First, using  Lemma~\ref{L:GT} we conclude that  if $g\in B_M$,
then $|u|\ll_{G,M} 1$. As a result, \eqref{E:local} gives that
 \begin{equation}\label{E:ooo1}
|u^n|\ll_{G,M} n^{k_1}
\end{equation}
 for every $g\in B_M$.
Next, notice that by Lemma~\ref{L:GT}     there exists $k_2=k_2(G)\in \N$
 such that for every $K>1$ we have
\begin{equation}\label{E:ooo2}
   K^{-k_2}|u-v|\leq d(g,h) \leq K^{k_2}|u-v|
   \end{equation}
for every $g,h\in G$  that satisfy  $|u|$, $|v|\leq K$ (remember that $u=\phi(g), v=\phi(h)\in \R^m$).
 Combining the estimates \eqref{E:local}, \eqref{E:ooo1}, and \eqref{E:ooo2}  we get
$$
 d(g^n,h^n)\ll_{G,M} n^{k_1k_2} |u^n-v^n|\ll_{G,M} n^{k_1+k_1k_2} |u-v|\ll_{G,M} n^{k_1+k_1k_2}d(g,h).
 $$
This establishes the advertised estimate with $k=k_1+k_1k_2$.
\end{proof}

\begin{proof}[Proof of Proposition~\ref{P:random}]
By Lemma~\ref{L:g^n} we get
that there exists $k_1=k_1(G)$ such that
\begin{equation}\label{L:g^nGamma}
d_G(g^n,h^n)\ll_{G,M} n^{k_1} d_G(g,h)
\end{equation}
 for every $g,h\in B_M$.
For every $K\in \N$ there exist
 $K^{m}$ points that form  an $1/K$-net for the set $[0,1)^m$ with
the  sup-norm. Combining this with    Lemma~\ref{L:GT} we get
that there exists $k_2=k_2(G)$ with the following property: for every $N\in \N$  there exists an  $1/N^{k_1+2}$
net of $B_M$ consisting of $N^{k_2}$ points.

Let $B_{N,M}$ be any such   $1/N^{k_1+2}$-net of $B_M$.
By construction,  $|B_{N,M}|=N^{k_2}$
for some  $k_2$ that depends only on $G$.
Furthermore, for every  $b\in B_M$
there exists $b_N\in B_{N,M}$ such that
 $d_G(b,b_N)\leq 1/N^{k_1+2}$.
 It follows from  \eqref{L:g^nGamma} that
$$\max_{1\leq n\leq N}d_G(b^n,b^n_N)\ll_{G,M} N^{k_1} d_G(b,b_N)\leq 1/N^2.$$
Therefore,
$$\max_{1\leq n\leq N}d_X(b^n\Gamma,b^n_N\Gamma)\ll_{G,M}  1/N^2.$$
Using this,
we deduce  \eqref{E:approximation} (with $o_{N\to\infty}(1)=\norm{c_n}_\infty\norm{F}_{\text{Lip}(X)}/N\leq 1/N$), completing the proof.
\end{proof}

\subsection{Proof of Theorem~\ref{T:nilsequence}}
We give a sketch of the proof of Theorem~\ref{T:nilsequence}. The missing details
 can be  extracted  from \cite{RW}.

Without loss of generality we can assume that  $\norm{F}_{\text{Lip}(X)}\leq 1$.

From Proposition~\ref{P:random} we conclude that there exists a $k=k(G,M)\in \N$ and a subset $B_{N,M}$ of $B_M$ with $|B_{N,M}|=N^k$ such that
  \begin{equation}\label{E:approximation'}
\max_{b\in B_M}\Big|\frac{1}{w(N)}\sum_{n=1}^N (X_n(\omega)-\sigma_n)
\ \! F(b^n\Gamma)\Big|= \max_{b\in B_{N,M}}\Big|\frac{1}{w(N)}\sum_{n=1}^N
(X_n(\omega)-\sigma_n)\ \! F(b^n\Gamma)\Big|+o_{N\to\infty}(1).
\end{equation}

Since the cardinality of $B_{N,M}$ is a power of $N$ that depends only on $G$ and $M$,
it follows that $|B_{N,M}|^{1/\log N}$ is bounded  by some  constant
that depends only on $G$ and $M$. Hence,
\begin{align}\label{E:hv1}
\norm{\max_{b\in B_{N,M}}\Big|\frac{1}{w(N)}\sum_{n=1}^N (X_n(\omega)-\sigma_n)\ \! F(b^n\Gamma)\Big|}_{L^{\log{N}}(\Omega)}
 &\ll_{G,M}\\ \notag \max_{b\in B_{N,M}} & \norm{\frac{1}{w(N)}\sum_{n=1}^N (X_n(\omega)-\sigma_n)\ \! F(b^n\Gamma)}_{L^{\log{N}}(\Omega)}.
\end{align}
Furthermore, arguing exactly as in \cite{RW} (pages 40-41), it can be  shown that for every sequence of complex numbers $(c_n)_{n\in\N}$ with $\norm{c_n}_\infty\leq 1$
one has
\begin{equation}\label{E:hv2}
\norm{\frac{1}{w(N)}\sum_{n=1}^N (X_n(\omega)-\sigma_n)\ \! c_n}_{L^{\log{N}}(\Omega)}
\ll
 \sqrt{\frac{\log{N}}{w(N)}}.
\end{equation}
Combining  \eqref{E:approximation'}, \eqref{E:hv1}, and \eqref{E:hv2} (with $c_n=F(b^n\Gamma)$), gives
\begin{equation}\label{E:hv3}
\norm{\max_{b\in B_M}\Big|\frac{1}{w(N)}\sum_{n=1}^N (X_n(\omega)-\sigma_n)\ \! F(b^n\Gamma)\Big|}_{L^{\log{N}}(\Omega)}
\ll_{G,M} \sqrt{\frac{\log{N}}{w(N)}}+o_{N\to\infty}(1).
\end{equation}

Next we  make use of the following simple lemma:
\begin{lemma}
Let $(Y_k)_{k\in\N}$ be a sequence of bounded, complex-valued random variables
 on a probability space $(\Omega,\Sigma, P)$.

Then almost surely we have
$$
\limsup_{k\to \infty} \frac{|Y_k(\omega)|}{\norm{Y_k}_{L^{\log{k}}(\Omega)}} \leq e
$$
where $e$ is the Euler number.
\end{lemma}
 Combining this lemma with \eqref{E:hv3}, we conclude that for  every nilmanifold $X=G/\Gamma$, with $G$ connected and
 simply connected,  and $F\in \text{Lip}(X)$  with $\norm{F}_{\text{Lip}(X)}\leq 1$,
  there exists a set $\Omega_{F,G,M}$
 of probability $1$, such that for every $\omega\in \Omega_{F,G,M}$ we have
$$
\max_{b\in B_M}\Big|\frac{1}{w(N)}\sum_{n=1}^N (X_n(\omega)-\sigma_n)\ \! F(b^n\Gamma)
\Big|\ll_{\omega, G,M} \sqrt{\frac{\log{N}}{w(N)}}+o_{N\to\infty}(1).
$$
Since by assumption $\log{N}/w(N)\to 0$, we get  that for  every nilmanifold $X$,  $F\in \text{Lip}(X)$, and $M>0$,
we almost surely have
$$
\lim_{N\to \infty}\max_{b\in B_M}\Big|\frac{1}{w(N)}\sum_{n=1}^N (X_n(\omega)-\sigma_n)\ \! F(b^n\Gamma)\Big|=0.
$$
This completes the proof of Theorem~\ref{T:nilsequence}, and finishes the proof of Theorem~\ref{T:random}.






\end{document}